\documentclass{amsart}
\usepackage{amsmath}
\usepackage{amsthm,amssymb}
\usepackage[unicode]{hyperref}
\hypersetup{
	colorlinks=true,
	linkcolor=black,
	citecolor=black,
	urlcolor=blue
}
\usepackage{xcolor}
\usepackage{mathtools}
\usepackage{enumitem}
\usepackage{xspace}
\usepackage{listings}
\usepackage{mathrsfs}

\lstdefinestyle{chattranscript}{
	basicstyle=\ttfamily\scriptsize,
	columns=fullflexible,
	keepspaces=true,
	showstringspaces=false,
	breaklines=true,
	breakatwhitespace=false,
	frame=single,
	xleftmargin=0.5em,
	xrightmargin=0.5em,
	aboveskip=0.8em,
	belowskip=0.8em,
	literate=
	{é}{{\'e}}1
	{ï}{{\"i}}1
	{ö}{{\"o}}1
	{–}{{--}}1
	{—}{{---}}1
	{“}{{``}}1
	{”}{{''}}1
	{’}{{'}}1
	{∎}{{$\square$}}1
}

\newtheorem{thm}{Theorem}[section]
\newtheorem{lem}[thm]{Lemma}
\newtheorem{prop}[thm]{Proposition}
\newtheorem{cor}[thm]{Corollary}

\theoremstyle{definition}
\newtheorem{defn}[thm]{Definition}
\newtheorem{rem}[thm]{Remark}
\newtheorem{fact}[thm]{Fact}

\numberwithin{equation}{section}

\newcommand{\bbq}{{\mathbb Q}}
\newcommand{\bbz}{{\mathbb Z}}
\newcommand{\bfpi}{{\boldsymbol \Pi}}
\newcommand{\bfsgm}{{\boldsymbol \Sigma}}
\newcommand{\bfdlt}{{\boldsymbol \Delta}}
\newcommand{\uphar}{\upharpoonright}
\newcommand{\Act}{\operatorname{Act}}
\newcommand{\Homeo}{\operatorname{Homeo}}
\newcommand{\Hom}{\operatorname{Hom}}
\newcommand{\NSub}{\operatorname{NSub}}
\newcommand{\Sub}{\operatorname{Sub}}
\newcommand{\STRP}{\ensuremath{\mathrm{STRP}}\xspace}

\title[STRP: Ascent, Complexity, and Obstructions]
{Strong Topological Rokhlin Property: Finite-Index Ascent,
	Descriptive Complexity, and Effective Obstructions}

\author{Jintao Luo}
\address{Independent Researcher, Chongqing, 400000, China}
\email{jintaoluo@foxmail.com}


\date{}

\begin{document}
	
	\begin{abstract}
		We give a finite symbolic characterization of the strong topological Rokhlin property in terms of globally realizable finite pattern systems. We use this characterization to prove finite-index ascent: if $H\leq G$ has finite index, $G$ is finitely generated, and $H$ has \STRP, then $G$ has \STRP. Consequently every finitely generated virtually free group has \STRP, whereas among countable locally virtually free groups \STRP holds exactly for the finitely generated ones. In the compact coding space $\NSub(F_\omega)$ of countable groups, the \STRP locus belongs to $\bfpi^0_4$, is $\bfsgm^0_3$-hard; the locus of finitely generated virtually free groups is $\bfsgm^0_3$-complete. For quotients $F_\omega/N$ with $N$ recursively enumerable, every projectively isolated subshift has decidable finite tuple-language and contains an $N$-recursive configuration. This yields effective SFT obstructions to \STRP, including subgroup and direct-product obstructions, and implies that $\mathrm{SL}_n(\bbq)$ does not have \STRP for any $n\geq2$.
	\end{abstract}
	
	\maketitle
	\section{Globally realizable triples}
	
	Let $\mathfrak C=2^{\omega}$ be the Cantor space. Fix a compatible metric $d$ on $\mathfrak C$. The group $\Homeo(\mathfrak C)$ carries the Polish topology induced by
	\[
	d_{\mathrm H}(u,v)=\max\left\{\sup_{x\in\mathfrak C}d(u(x),v(x)),\ \sup_{x\in\mathfrak C}d(u^{-1}(x),v^{-1}(x))\right\}.
	\]
	For a countable group $G$, the space of continuous Cantor actions is
	$\Act_G(\mathfrak C)=\Hom(G,\Homeo(\mathfrak C))$, viewed as a closed subspace of $\Homeo(\mathfrak C)^G$. The conjugation action is given by $(u\cdot\alpha)(g)=u\alpha(g)u^{-1}$. A countable group $G$ has the \textbf{strong topological Rokhlin property}, abbreviated \STRP, if the conjugation action $\Homeo(\mathfrak C)\curvearrowright\Act_G(\mathfrak C)$ has a comeager orbit.
	
	We next fix the symbolic-dynamical terminology used throughout the paper. Let $A$ be a finite nonempty set, called an \textbf{alphabet}. The \textbf{full $G$-shift} is $A^G$ with the product topology and the left shift action
	$(\sigma_G(g)x)(h)=x(g^{-1}h)$. A \textbf{$G$-subshift} is a nonempty closed $G$-invariant subset of $A^G$. If $F\subseteq G$ is finite, an element of $A^F$ is an $F$-\textbf{pattern}. For a subshift $X\subseteq A^G$, its $F$-\textbf{language} is $X_F=\{x\uphar F:x\in X\}$.
	
	For a finite $F\subseteq G$ and $\mathcal F\subseteq A^F$, write
	\[
	[F,\mathcal F]_G=\{x\in A^G:(\sigma_G(g)x)\uphar F\in\mathcal F\text{ for every }g\in G\}.
	\]
	A subshift is a \textbf{subshift of finite type}, or an \textbf{SFT}, if it has the form $[F,\mathcal F]_G\neq\emptyset$ for some finite $F$ and $\mathcal F$. This allowed-pattern convention is equivalent to the customary finite forbidden-pattern definition.
	
	Let $\Sub_G(A)$ denote the space of nonempty $G$-subshifts over $A$. For $X\in\Sub_G(A)$ and finite $F\subseteq G$, define
	$\mathcal N_X^F=\{Y\in\Sub_G(A):Y_F=X_F\}$. The families $\mathcal N_X^F$ form a clopen basis of a compact zero-dimensional metrizable topology on $\Sub_G(A)$.
	
\begin{fact}\label{fact:normalization}
	Suppose that $X=[F,\mathcal F]_G$ is nonempty. Then $X=[F,X_F]_G$. If $K\supseteq F$ is finite, then $X=[K,X_K]_G$.
\end{fact}

\begin{proof}
	Since $x\in X$ implies $x\uphar F\in\mathcal F$, one has $X_F\subseteq\mathcal F$. Thus $x\in[F,X_F]_G$ implies $(\sigma_G(g)x)\uphar F\in X_F\subseteq\mathcal F$ for every $g\in G$, so $x\in X$. Conversely, if $x\in X$, then $\sigma_G(g)x\in X$ for every $g\in G$, whence $(\sigma_G(g)x)\uphar F\in X_F$ for every $g\in G$. Thus $x\in[F,X_F]_G$.
	
	Let $K\supseteq F$ be finite. If $x\in X$, then $\sigma_G(g)x\in X$ for every $g\in G$, so $(\sigma_G(g)x)\uphar K\in X_K$ for every $g\in G$ and $x\in[K,X_K]_G$. Conversely, if $x\in[K,X_K]_G$, then for every $g\in G$,
	$(\sigma_G(g)x)\uphar K\in X_K$, and therefore
	$(\sigma_G(g)x)\uphar F=((\sigma_G(g)x)\uphar K)\uphar F\in X_F$.
	Thus $x\in[F,X_F]_G=X$.
\end{proof}

	\begin{defn}\label{def:globally-realizable}
		A triple $(A,F,\mathcal F)$ is \textbf{globally realizable over $G$} if $A$ is a finite alphabet, $F\subseteq G$ is finite, $[F,\mathcal F]_G$ is nonempty, and
		$([F,\mathcal F]_G)_F=\mathcal F$.
	\end{defn}
	
	Every locally allowed pattern in a globally realizable triple occurs in a global configuration. If $X\subseteq A^G$ is a nonempty subshift, then $(A,F,X_F)$ is globally realizable for every finite $F\subseteq G$: the inclusion $X\subseteq[F,X_F]_G$ gives $X_F\subseteq([F,X_F]_G)_F$, while the defining condition at the identity gives the reverse inclusion.
	
	A map $p:B\to A$ between finite alphabets induces the \textbf{one-block map} $p^G:B^G\to A^G$ defined by $p^G(y)(g)=p(y(g))$. A continuous $G$-equivariant map between subshifts is a \textbf{block map}, or a \textbf{morphism}; it is a \textbf{factor map} when it is onto. By the Curtis--Hedlund--Lyndon theorem, every morphism is determined by a finite local rule \cite[Corollary~2.9]{doucha2024}. A subshift is \textbf{sofic} if it is a factor of an SFT.
	
	\begin{defn}[{\cite[Definition~2.22]{doucha2024}}]\label{def:projectively-isolated}
		A subshift $X\in\Sub_G(A)$, where $|A|\geq2$, is \textbf{projectively isolated} if there are a finite alphabet $B$ with $|B|\geq2$, a subshift $Y\in\Sub_G(B)$, a finite set $F\subseteq G$, and a continuous $G$-equivariant map $\phi:Y\to X$ such that $ \{Z\in\mathcal N_Y^F:\phi[Z]=X\} $ has nonempty interior containing $Y$. For $Z\in\mathcal N_Y^F$, the notation $\phi[Z]$ means the image obtained by applying to $Z$ the same finite local rule that defines $\phi$ on $Y$; \cite[Lemma~2.21]{doucha2024} proves that the rule is defined on every such $Z$ after enlarging $F$ if necessary.
	\end{defn}
	
	SFTs are dense in every $\Sub_G(B)$. A projective-isolation witness can be recoded with a one-block factor map by \cite[Lemma~2.27]{doucha2024}.
	
	\begin{prop}\label{prop:projectively-isolated}
		Let $X\subseteq A^G$ be a subshift with $|A|\geq2$. The following are equivalent.
		\begin{enumerate}[label=(\arabic*),font=\upshape]
			\item The subshift $X$ is projectively isolated.
			\item There are a finite alphabet $B$, a finite $F\subseteq G$, a globally realizable triple $(B,F,\mathcal F)$, and a surjection $p:B\to A$ such that
			\begin{equation}\label{eq:s1-02}
				\forall Z\in\Sub_G(B)\ 
				[Z_F=\mathcal F\longrightarrow p^G(Z)=X].
			\end{equation}
		\end{enumerate}
	\end{prop}
	
	\begin{proof}
		Assume (2), and let $Y=[F,\mathcal F]_G$. Global realizability gives $Y_F=\mathcal F$, so \eqref{eq:s1-02} with $Z=Y$ gives $p^G(Y)=X$. Moreover, $Z\in\mathcal N_Y^F$ implies $Z_F=Y_F=\mathcal F$, and hence $p^G(Z)=X$ by \eqref{eq:s1-02}. Thus $\{Z\in\mathcal N_Y^F:p^G(Z)=X\}=\mathcal N_Y^F$, and Definition~\ref{def:projectively-isolated} applies to $p^G\upharpoonright Y:Y\to X$.
		
		Assume (1). By \cite[Lemma~2.27]{doucha2024}, choose a projective-isolation witness $Y_0\subseteq B_0^G$ for $X$ whose factor map is the one-block map induced by some $p_0:B_0\to A$. Thus there is an open neighborhood $\mathcal U$ of $Y_0$ such that $Z\in\mathcal U$ implies $p_0^G(Z)=X$. Choose a finite $K\subseteq G$ such that $\mathcal N_{Y_0}^K\subseteq\mathcal U$. By density of SFTs, choose an SFT $Y\in\mathcal N_{Y_0}^K$. Write $Y=[E,\mathcal E]_G$, let $F=E\cup K\cup\{1_G\}$, and define $\mathcal F=Y_F$. Fact~\ref{fact:normalization} gives $Y=[F,\mathcal F]_G$.
		
		Extend $B_0$, if necessary, to a finite alphabet $B\supseteq B_0$ and extend $p_0$ to a surjection $p:B\to A$. Regard $\mathcal F$ as a subset of $B^F$. Now work in alphabet $ B $. Since $\mathcal F\subseteq B_0^F$ and $1_G\in F$, if $y\in[F,\mathcal F]_G$, then $y(g)=(\sigma_G(g^{-1})y)(1_G)\in B_0$ for every $g\in G$. Thus $[F,\mathcal F]_G\subseteq B_0^G$, so $[F,\mathcal F]_G=Y$. Since $Y_F=\mathcal F \subseteq B^F$ as well, the triple $(B,F,\mathcal F)$ is globally realizable.
		
		Let $Z\in\Sub_G(B)$ satisfy $Z_F=\mathcal F$. Since $1_G\in F$ and $\mathcal F\subseteq B_0^F$, the same argument gives $Z\subseteq B_0^G$. Since $K\subseteq F$, one has $Z_K=(Z_F)\uphar K=\mathcal F\uphar K=Y_K=(Y_0)_K$. Thus $Z\in\mathcal N_{Y_0}^K\subseteq\mathcal U$, so $p^G(Z)=p_0^G(Z)=X$. This proves \eqref{eq:s1-02}.
	\end{proof}
	
	Doucha's symbolic characterization states that $G$ has \STRP exactly when projectively isolated subshifts are dense in $\Sub_G(A)$ for every finite alphabet $A$ with $|A|\geq2$ \cite[Theorem~3.1]{doucha2024}. The next theorem replaces the quantification over arbitrary subshifts and neighborhoods by finite data.
	
	\begin{thm}\label{thm:finite-characterization}
		Let $G$ be a countable group. The following are equivalent.
		\begin{enumerate}[label=(\arabic*),font=\upshape]
			\item The group $G$ has \STRP.
			\item For every globally realizable triple $(A,E,\mathcal E)$ with $|A|\geq2$, there exist a globally realizable triple $(B,F,\mathcal F)$ and a surjection $p:B\to A$ such that
			\begin{equation}\label{eq:s1-03}
				\bigl(p^G([F,\mathcal F]_G)\bigr)_E=\mathcal E
			\end{equation}
			and
			\begin{equation}\label{eq:s1-04}
				\forall Z\in\Sub_G(B)\
				[Z_F=\mathcal F\longrightarrow
				p^G(Z)=p^G([F,\mathcal F]_G)].
			\end{equation}
		\end{enumerate}
	\end{thm}
	
	\begin{proof}
		Assume (1), and let $(A,E,\mathcal E)$ be globally realizable with $|A|\geq2$. Write $X_0=[E,\mathcal E]_G$. Then $(X_0)_E=\mathcal E$. By \cite[Theorem~3.1]{doucha2024}, choose a projectively isolated $X\in\mathcal N_{X_0}^E$. Proposition~\ref{prop:projectively-isolated} gives a globally realizable triple $(B,F,\mathcal F)$ and a surjection $p:B\to A$ such that $Z_F=\mathcal F$ implies $p^G(Z)=X$ for every $Z\in\Sub_G(B)$. Since $[F,\mathcal F]_G$ has $F$-language $\mathcal F$, one has $p^G([F,\mathcal F]_G)=X$. Thus $\bigl(p^G([F,\mathcal F]_G)\bigr)_E=X_E=(X_0)_E=\mathcal E$, which is \eqref{eq:s1-03}, while \eqref{eq:s1-04} follows from the same implication.
		
		Assume (2). Let $X\in\Sub_G(A)$ with $|A|\geq2$, and let $E\subseteq G$ be finite. Since $X\subseteq[E,X_E]_G$ and $([E,X_E]_G)_E\subseteq X_E$, the triple $(A,E,X_E)$ is globally realizable. Apply (2), and write $V=[F,\mathcal F]_G$ and $Y=p^G(V)$. Equation~\eqref{eq:s1-04} says that $Z_F=\mathcal F$ implies $p^G(Z)=Y$ for every $Z\in\Sub_G(B)$, so Proposition~\ref{prop:projectively-isolated} gives that $Y$ is projectively isolated. Equation~\eqref{eq:s1-03} gives $Y_E=X_E$, and hence $Y\in\mathcal N_X^E$. Thus every basic neighborhood of every $X\in\Sub_G(A)$ contains a projectively isolated subshift. By \cite[Theorem~3.1]{doucha2024}, $G$ has \STRP.
	\end{proof}

	\section{Finite-index ascent and local virtual freeness}
	
	\subsection{Finite-index ascent}
	
	Let $H\leq G$. The \textbf{index} $[G:H]$ is the cardinality of the left-coset space $G/H$, equivalently of the right-coset space $H\backslash G$. The subgroup has \textbf{finite index} when this number is finite. Two symbolic constructions for finite-index inclusions are used below: higher power, which records all right-coset coordinates at one $H$-coordinate, and free extension, which imposes the same $H$-rules independently on each left coset. These constructions occur in \cite[Section~3]{bitar2024}; the higher-block form also appears in \cite[Section~3.1]{carrollpenland2015}. We record the exact formulas needed here.
	
	\begin{lem}\label{lem:higher-power}
		Let $H\leq G$ have finite index, let $\mathcal R=H\backslash G$, and choose $t_r\in G$ with $Ht_r=r$ for each $r\in\mathcal R$. For a finite alphabet $A$, write $C=A^{\mathcal R}$. The map $\Phi_A:A^G\to C^H$ defined by
		$\Phi_A(x)(h)(r)=x(ht_r)$ is an $H$-equivariant homeomorphism.
	\end{lem}
	
	\begin{proof}
		Every $g\in G$ has a unique expression $g=ht_r$ with $h\in H$ and $r\in\mathcal R$. The inverse map is
		$\Phi_A^{-1}(u)(ht_r)=u(h)(r)$, which is well-defined: we have that
		\[
		\begin{aligned}
			\Phi^{-1}_A(\Phi_A(x))(ht_r) &= \Phi_A(x)(h)(r) = x(ht_r),\\
			 \Phi_A(\Phi^{-1}_A (u))(h)(r) &= \Phi^{-1}_A(u)(ht_r) = u(h)(r).
		\end{aligned}
		\]
		Both maps are coordinate reindexings and are continuous:
		\[
		\begin{aligned}
			\Phi_A(x)(h) = m &\text{ if and only if } \forall r \in \mathcal R  \ [x(ht_r) = m(r)],\\
			\Phi^{-1}_A(u)(ht_r) = a &\text{ if and only if } \exists m\in C \ [m(r) =  a \ \wedge \ u(h) = m].
		\end{aligned}
		\]
		
		For $a,h\in H$ and $r\in\mathcal R$,
		\[
		\Phi_A(\sigma_G(a)x)(h)(r)
		=x(a^{-1}ht_r)
		=(\sigma_H(a)\Phi_A(x))(h)(r).
		\]
		This is $H$-equivariance.
	\end{proof}
	
	\begin{lem}[{\cite[Lemma~3.4]{bitar2024}}]\label{lem:free-extension}
		Let $H\leq G$, let $B$ be a finite alphabet, and let $V=[F,\mathcal F]_H\subseteq B^H$ be nonempty, where $F\subseteq H$ is finite. Its \textbf{free extension} to $G$ is $V^{\uparrow}=[F,\mathcal F]_G$. Fix representatives $L\subseteq G$ for the left cosets $G/H$. For $y\in B^G$ and $\ell\in L$, define $v_\ell\in B^H$ by $v_\ell(h)=y(\ell h)$. Then
		\begin{equation}\label{eq:s2-01}
			y\in V^{\uparrow}\ \longleftrightarrow\ \forall \ell \in L \ [ v_\ell\in V].
		\end{equation}
		In particular, $V^{\uparrow}$ is nonempty, and its restrictions to distinct left cosets may be chosen independently.
	\end{lem}
	
	\begin{proof}
		Suppose $y\in V^{\uparrow}$. For $\ell\in L$, we show that $\forall a\in H \ [\sigma_H (a) v_\ell \upharpoonright F \in \mathcal F]$; but it's $(\sigma_G(a\ell^{-1})y)\upharpoonright F$, which is in $ \mathcal F $. So $v_\ell\in V$.
		
		Conversely, if for every $ \ell\in L $ we have $ v_\ell\in V $ i.e. $\forall a\in H \ [\sigma_H (a) v_\ell \upharpoonright F \in \mathcal F]$, then for any $ g\in G $, write $g^{-1}=\ell a^{-1}$ with $\ell\in L$ and $a\in H$. Then we have $ (\sigma_G(g)y) \upharpoonright F = (\sigma_H(a)v_\ell)\upharpoonright F \in \mathcal{F}$. So $ y\in V^{\uparrow} $.
	\end{proof}
	
	\begin{thm}\label{thm:finite-index}
		Let $G$ be a finitely generated countable group and let $H\leq G$ have finite index. If $H$ has \STRP, then $G$ has \STRP.
	\end{thm}
	
	\begin{proof}
		We verify condition (2) of Theorem~\ref{thm:finite-characterization}. Fix a globally realizable triple $(A,E,\mathcal E)$ over $G$, and write $X=[E,\mathcal E]_G$. Then $X_E=\mathcal E$.
		
		\smallskip
		\noindent\emph{Step 1: encode $X$ as an $H$-subshift.}
		Let $\mathcal R=H\backslash G$. Choose representatives $t_r$ with $t_H=1_G$, write $C=A^{\mathcal R}$, and let $W=\Phi_A(X)\subseteq C^H$. For $r\in\mathcal R$ and $e\in E$, the unique element $\lambda(r,e)\in H$ satisfying
		\begin{equation}\label{eq:s2-02}
			t_re=\lambda(r,e)t_{re}
		\end{equation}
		is well defined, where $re$ denotes the right action of $G$ on $H\backslash G$. Let
		$K=\{1_H\}\cup\{\lambda(r,e):r\in\mathcal R,\ e\in E\}$. The triple $(C,K,W_K)$ is globally realizable because $\emptyset \neq W\subseteq [K, W_K]_H$ by $ H $-equivariance, and $ ([K,W_K]_{H})_K \subseteq W_K $ by shift-invariance for $ 1_H $. Also $|C|\geq2$.
		
		\smallskip
		\noindent\emph{Step 2: use \STRP for $H$.}
		Theorem~\ref{thm:finite-characterization}, applied to $(C,K,W_K)$, gives a globally realizable triple $(B,F,\mathcal F)$ over $H$ and a surjection $p:B\to C$. Write $ V=[F,\mathcal F]_H,\ P=p^H(V) $. The conclusion of that theorem is $V_F=\mathcal F,\ P_K=W_K $, and for every $H$-subshift $U\subseteq B^H$,
		\begin{equation}\label{eq:s2-05}
			U_F=\mathcal F\ \longrightarrow\ p^H(U)=P.
		\end{equation}
		
		\smallskip
		\noindent\emph{Step 3: encode the right-coset phase by finite rules.}
		Let $V^{\uparrow}=[F,\mathcal F]_G$. Lemma~\ref{lem:free-extension} gives nonemptiness. Choose a finite symmetric generating set $S$ of $G$. For $r\in\mathcal R$, define $m_r\in\mathcal R^G$ by $m_r(g)=rg$. Let $Q_M=\{1_G\}\cup S$ and
		\begin{equation}\label{eq:s2-06}
			\mathcal M=\{a\in\mathcal R^{Q_M}:  \forall s\in S \ [a(s)=a(1_G)s]\}.
		\end{equation}
		Then $[Q_M,\mathcal M]_G=\{m_r:r\in\mathcal R\}$. Indeed, the translated local equations give $m(gs)=m(g)s$ for every $g\in G$ and $s\in S$. Induction along a word over $S$ gives $m(g)=m(1_G)g$. This is the only point at which finite generation of $G$ is used.
		
		\smallskip
		\noindent\emph{Step 4: build the lifting SFT.}
		Let $D=\mathcal R\times B\times B$ and
		$Q_0=Q_M\cup F\cup\{t_r^{-1}:r\in\mathcal R\}$. For $a\in D^{Q_0}$, write $a=(a_1,a_2,a_3)$ coordinatewise. Define
		\[
		\mathcal Q_0 = \{a\in D^{Q_0}: \forall s\in S\ [a_1(s)=a_1(1_G)s] \ \wedge \  a_2\uphar F \in\mathcal F \ \wedge \ a_3(1_G)=a_2 (t_{a_1(1_G)}^{-1})\}.
		\]
		Write $\widehat V=[Q_0,\mathcal Q_0]_G$. We have that
		\[
		\widehat V = \{(m,y,d)\in D^G: m\in\{m_r:r\in\mathcal R\}\ \wedge 
		y\in V^{\uparrow} \ \wedge \ \forall g\in G \ [d(g)=y(gt_{m(g)}^{-1})]\}.
		\]
		The SFT $\widehat V$ is nonempty. Let $Q=Q_0$ and $\mathcal Q=(\widehat V)_Q$. Fact~\ref{fact:normalization} gives that $\widehat V=[Q,\mathcal Q]_G$ and $(D,Q,\mathcal Q)$ is globally realizable.
		
		Define $q:D\to A$ by $ q(r,b,c)=p(c)(r) $. This map is onto. Given $a\in A$, choose $r\in\mathcal R$ and $u\in C$ with $u(r)=a$, choose $c\in B$ with $p(c)=u$, and choose any $b\in B$.
		
		\smallskip
		\noindent\emph{Step 5: analyze every subshift with $Q$-language $\mathcal Q$.}
		Let $Z\subseteq D^G$ be a subshift with $Z_Q=\mathcal Q$. Since $\widehat V=[Q,\mathcal Q]_G$, one has $Z\subseteq\widehat V$. For $r\in\mathcal R$, let
		\begin{equation}\label{eq:s2-10}
			Z_r=\{z\in Z: z_1 = m_r\},
			\
			U_Z=\{z_2\uphar H:z \in Z_H\}.
		\end{equation}

		For $r=H$, the set $Z_H$ is compact and invariant under $H$. Restriction of the second coordinate from $Z_H$ to $H$ is continuous and $H$-equivariant. Its image $U_Z$ is a nonempty $H$-subshift. For any $ y\upharpoonright H \in U_Z $, we have that $ y\in V^{\uparrow} $. By applying Lemma \ref{lem:free-extension} with the representative for $ H $ being $ 1_G $, we have that $ y\upharpoonright H \in V $, so $ U_Z\subseteq V $.
		
		We claim that $(U_Z)_F=\mathcal F$. The inclusion $(U_Z)_F\subseteq\mathcal F$ follows from $U_Z\subseteq V$. For the reverse inclusion, fix $a\in\mathcal F$. Since $V_F=\mathcal F$, choose $v\in V$ with $v\uphar F=a$. By Lemma~\ref{lem:free-extension}, we can freely define a $y_0\in V^{\uparrow}$ with $y_0\uphar H=v$. Define $\bar y_0$ by $ \forall g\in G $ letting  $ \bar{y_0}(g) = y_0(gt^{-1}_{m_H (g)}) $, and let $z_0=(m_H,y_0,\bar y_0)$, which is in $\widehat V$. The equality $Z_Q=(\widehat V)_Q$ gives $z\in Z$ with $z\uphar Q=z_0\uphar Q$. As $ 1_G\in Q $, $ z_1(1_G) = m_H(1_G) = H $; and as $ Z\subseteq \widehat V $, $ z_1 = m_H $, which implies $ z\in Z_H $. Also, $ z_2 \upharpoonright F = y_0 \upharpoonright F = a $. This proves that $(U_Z)_F=\mathcal F$. Then Equation \eqref{eq:s2-05} gives
		\begin{equation}\label{eq:s2-11}
			p^H(U_Z)=P.
		\end{equation}
		
		\smallskip
		\noindent\emph{Step 6: compute the image and its $E$-language.}
		Take $z=(m_H,y,d)\in Z_H$ and let $v=y\uphar H$, so $ v\in U_Z $. If $g=ht_r$, then $m_H(g)=r$ and $d(g)=y(ht_r t^{-1}_{m_H(g)})=v(h)$. The one-block image is
		\begin{equation}\label{eq:s2-12}
			q^G(z)(ht_r)=q(r, y(g), v(h))=p(v(h))(r)=\Phi_A^{-1}(p^H(v))(ht_r).
		\end{equation}
		Equation \eqref{eq:s2-11} yields
		$q^G(Z_H)\subseteq \Phi_A^{-1}(P)$. Also, for any $ u\in P $, it's from some $ z\in Z_H $ with $ z_2\upharpoonright H \in U_Z $ such that $ u = p^H(z_2\upharpoonright H) $, and \eqref{eq:s2-12} applies, so $q^G(Z_H) =  \Phi_A^{-1}(P)$. 
		
		For every $ r\in \mathcal R $, define $ \widehat{V}_r = \{(m,y,d)\in \widehat{V}: m=m_r\} $. The same calculation, using $p^H(V)=P$ and the independent free extension of every $v\in V$, gives
		$q^G(\widehat V_H)=\Phi_A^{-1}(P)$.
		
		Now we show that for any $ r\in \mathcal{R} $, $ q^G(Z_r) = \sigma_G(t_r^{-1})\Phi_A^{-1}(P) = q^G(\widehat{V}_r) $. This is because $ Z_r = \sigma_G(t^{-1}_r)Z_H $, $ \widehat{V}_r = \sigma_G(t_r^{-1})\widehat{V}_H $ and $ q^G $ is equivariant. So
		\begin{equation}\label{eq:s2-13}
			q^G(Z)=\bigcup_{r\in\mathcal R}\sigma_G(t_r^{-1})\Phi_A^{-1}(P)
			=q^G(\widehat V).
		\end{equation}
		
		For $r\in\mathcal R$, define $\Theta_r:C^K\to A^E$ by
		$\Theta_r(\eta)(e)=\eta(\lambda(r,e))(re)$. Equation \eqref{eq:s2-02} gives, for $u\in P$ and $e\in E$,
		\begin{equation}\label{eq:s2-14}
			\bigl(\sigma_G(t_r^{-1})\Phi_A^{-1}(u)\bigr)(e)=\Phi^{-1}_A(u)(t_r e) = \Phi^{-1}_A(u)(\lambda(r,e) t_{re})
			=u(\lambda(r,e))(re).
		\end{equation}
	So $ \sigma_G(t_r^{-1})\Phi_A^{-1}(u) \uphar E= \Theta_r(u\uphar K) $; and by $P_K=W_K$, $W=\Phi_A(X)$,
		\begin{equation}\label{eq:s2-15}
			\begin{aligned}
				\bigl(\sigma_G(t_r^{-1})\Phi_A^{-1}(P)\bigr)_E
				&=\Theta_r(P_K)=\Theta_r(W_K)\\ &=\bigl(\sigma_G(t_r^{-1})\Phi_A^{-1}(W)\bigr)_E \\
				&=\bigl(\sigma_G(t_r^{-1})X\bigr)_E=X_E=\mathcal E.
			\end{aligned}
		\end{equation}
		Each phase in \eqref{eq:s2-13} has $E$-language $\mathcal E$. The union in \eqref{eq:s2-13} satisfies
		$(q^G(\widehat V))_E=\mathcal E$, and \eqref{eq:s2-13} is the universal image property. The triple $(D,Q,\mathcal Q)$ and the surjection $q$ satisfy Theorem~\ref{thm:finite-characterization}. The group $G$ has \STRP.
	\end{proof}
	
	\begin{rem}
		Theorem~\ref{thm:finite-index} was also obtained independently by Xu; see \cite[Theorem~7.3]{xu2026}.
	\end{rem}
	
	A group is \textbf{virtually free} if it contains a free subgroup of finite index. We include finite groups by allowing the free subgroup to have rank zero.
	
	\begin{cor}\label{cor:virtually-free}
		Every finitely generated virtually free group has \STRP.
	\end{cor}
	
	\begin{proof}
		Let $G$ be infinite and finitely generated virtually free. It contains a finite-index free subgroup $F_m$ of finite positive rank. Kwiatkowska proved that the diagonal conjugation action of $\Homeo(\mathfrak C)$ on $\Homeo(\mathfrak C)^m$ has a comeager orbit \cite[Theorem~1.1]{kwiatkowska2012}. Evaluation on a free basis is an equivariant homeomorphism $\Act_{F_m}(\mathfrak C)\cong\Homeo(\mathfrak C)^m$. The group $F_m$ has \STRP. Theorem~\ref{thm:finite-index} applies.
		
		If $G$ is finite, then $A^G$ is finite for every finite alphabet $A$. The space $\Sub_G(A)$ is finite and discrete. Every subshift is isolated, and every isolated subshift is projectively isolated by \cite[Lemma~2.23]{doucha2024}. Doucha's characterization \cite[Theorem~3.1]{doucha2024} gives \STRP.
	\end{proof}
	
	\begin{cor}\label{cor:sl2z}
		The group $\mathrm{SL}_2(\bbz)$ has \STRP.
	\end{cor}
	
	\begin{proof}
		The classical splitting
		$\mathrm{SL}_2(\bbz)\cong C_4*_{C_2}C_6$ makes $\mathrm{SL}_2(\bbz)$ virtually free. More precisely, its derived subgroup is free of rank two and has index $12$; see \cite[p.~1525]{andre2021}. Corollary~\ref{cor:virtually-free} applies.
	\end{proof}
	
	\begin{rem}\label{rem:sl34z}
		The corresponding conclusion fails for $\mathrm{SL}_n(\bbz)$ whenever
		$n\geq3$. For $n=3$, it's well known that $ \bbz^2\leq \mathrm{SL}_3(\bbz) $: let
		$u=I_3+E_{12}$ and $v=I_3+E_{13}$. Since
		$E_{12}E_{13}=E_{13}E_{12}=0$, one has
		$u^av^b=I_3+aE_{12}+bE_{13}$ for all $a,b\in\bbz$. Hence
		$(a,b)\mapsto u^av^b$ is injective and
		$\langle u,v\rangle\cong\bbz^2$. The standard block-diagonal embedding
		\[
		\mathrm{SL}_n(\bbz)\longrightarrow\mathrm{SL}_{n+1}(\bbz),
		\
		A\longmapsto
		\begin{pmatrix}
			A&0\\
			0&1
		\end{pmatrix},
		\]
		shows in particular that
		$\bbz^2\leq\mathrm{SL}_3(\bbz)\leq\mathrm{SL}_4(\bbz)$.
		
		Now \cite[Theorem~5.8]{barbiericarrasco2024} gives a nonempty
		$\bbz^2$-SFT of nonzero Medvedev degree, and
		\cite[Corollary~4.4]{barbiericarrasco2024} transfers the existence of
		such an SFT to finitely generated overgroups. Since
		$\mathrm{SL}_3(\bbz)$ and $\mathrm{SL}_4(\bbz)$ are finitely generated
		and recursively presented, \cite[Theorem~1.1]{carrasco2026} implies
		that neither group has \STRP.
		
		For $n\geq5$, this conclusion also follows from the simulation
		results as is shown in the literature \cite{carrasco2026}. Barbieri and Carrasco-Vargas record
		that their simulation theorem applies to $\mathrm{SL}_n(\bbz)$ in these
		dimensions \cite[Introduction]{barbiericarrasco2024}. Their
		\cite[Proposition~5.14(2)]{barbiericarrasco2024} gives a nonempty
		$\mathrm{SL}_n(\bbz)$-SFT of nonzero Medvedev degree. The decidable word
		problem makes $\mathrm{SL}_n(\bbz)$ recursively presented, and
		\cite[Theorem~1.1]{carrasco2026} shows that it does not have \STRP.
	\end{rem}

	\subsection{Locally virtually free groups}
	
	A group is \textbf{locally virtually free} if each of its finitely generated
	subgroups is virtually free. Doucha gave broad obstructions to \STRP for
	non-finitely-generated groups \cite[Theorem~5.5]{doucha2024}. More recently,
	Doucha, Melleray, and Tsankov proved that, for a non-finitely-generated group,
	the projectively isolated subshifts are exactly the minimal sofic subshifts
	\cite[Theorem~7.14]{douchamelleraytsankov2026}. Our argument uses two further
	ideas from their work: the coinduction and minimality arguments of
	\cite[Section~7]{douchamelleraytsankov2026}, and the special-symbol construction
	of \cite[Section~8.2]{douchamelleraytsankov2026}. We combine these with
	Minasyan's theorem that every virtually free group has property (LR), meaning
	that every finitely generated subgroup is a retract of a finite-index subgroup
	\cite[Theorem~1.1]{minasyan2026}, and with the special-symbol construction for
	hyperbolic groups due to Dahmani and Yaman
	\cite[Proposition~4.1]{dahmaniyaman2008}.
	
	For $H\leq K$ and a subshift $X\subseteq A^K$, write $X\uphar H=\{x\uphar H:x\in X\}$. If $T\subseteq K$ is a transversal for the left cosets $K/H$, define $x_t\in A^H$ by $x_t(h)=x(th)$ for $t\in T$ and $h\in H$, as usual. We say that $X$ is \emph{coinduced from $X\uphar H$} if
	\[
	X=\{x\in A^K:\forall t\in T\ [x_t\in X\uphar H]\}.
	\]
	This condition is independent of the choice of $T$.
	
	The following lemma makes explicit the coinduction mechanism used in the proof
	of \cite[Theorem~7.14]{douchamelleraytsankov2026}; compare also
	\cite[Proposition~7.16]{douchamelleraytsankov2026}.
	
	\begin{lem}\label{lem:sofic-coinduction}
		Let $X\subseteq A^G$ be sofic. There is a finitely generated subgroup $H\leq G$ such that, for every $H\leq K\leq G$, the subshift $X\uphar K$ is coinduced from $X\uphar H$, and $X$ is coinduced from $X\uphar K$.
	\end{lem}
	
	\begin{proof}
		After enlarging the alphabet if necessary, write $X=p^G(Y)$, where $Y=[F,\mathcal F]_G\subseteq B^G$ is an SFT and $p:B\to A$ is a surjection. Let $H=\langle F\rangle$ and $V=[F,\mathcal F]_H$. By Lemma~\ref{lem:free-extension}, $Y$ is the free extension of $V$, so $Y\uphar H=V$.
		
		Also, fix a transversal $T\subseteq G$ for the left cosets $G/H$. For $y\in Y$ and $t\in T$, let $y_t(h)=y(th)$ for $h\in H$. Lemma~\ref{lem:free-extension} gives
		\[
		y\in Y\quad\longleftrightarrow\quad
		\forall t\in T\ [y_t\in V],
		\]
		and the family $(y_t)_{t\in T}$ may be chosen independently. Since $p^G$ is one-block, $(p^G(y))_t=p^H(y_t)$ for every $t\in T$. Thus
		\[
		x\in X\quad\longleftrightarrow\quad
		\forall t\in T\ [x_t\in p^H(V)],
		\]
		and $X\uphar H=p^H(V)$. Hence $X$ is coinduced from $X\uphar H$.
		
		Let $H\leq K\leq G$. Choose transversals $R\subseteq K$ for $K/H$ and $S\subseteq G$ for $G/K$, with $1_G\in S$. Then $\{sr:s\in S,\ r\in R\}$ is a transversal for $G/H$. For $x\in A^G$, $s\in S$, and $r\in R$, write $x_s\in A^K$ and $x_{sr}\in A^H$ for the corresponding slices. The preceding characterization gives
		\[
		x\in X
		\ \longleftrightarrow\ 
		\forall s\in S\ \forall r\in R\ [x_{sr}\in X\uphar H].
		\]
		Since $X\uphar H\neq\emptyset$, projecting to the $s=1_G$ slice gives
		\[
		X\uphar K
		=
		\{u\in A^K:\forall r\in R\ [u_r\in X\uphar H]\}.
		\]
		Indeed, every $x\in X$ satisfies the displayed condition on $x_{1_G}$, while any such $u$ extends to a point of $X$ by choosing arbitrary members of $X\uphar H$ on the remaining $srH$-cosets. Since $(x_s)_r=x_{sr}$, we obtain
		\[
		x\in X
		\ \longleftrightarrow\ 
		\forall s\in S\ [x_s\in X\uphar K].
		\]
		Thus $X\uphar K$ is coinduced from $X\uphar H$, and $X$ is coinduced from $X\uphar K$.
	\end{proof}
	
	\begin{thm}\label{thm:locally-virtually-free}
		Let $G$ be a countable locally virtually free group. Then $G$ has \STRP if and only if $G$ is finitely generated.
	\end{thm}
	
	\begin{proof}
		If $G$ is finitely generated, then it is virtually free, and Corollary~\ref{cor:virtually-free} gives \STRP. Assume that $G$ is not finitely generated.
		
		First suppose that every finitely generated $H\leq G$ is contained in a finitely generated $K\leq G$ with $[K:H]=\infty$. Let $X$ be projectively isolated. By \cite[Theorem~7.14]{douchamelleraytsankov2026}, $X$ is minimal and sofic. Choose $H$ as in Lemma~\ref{lem:sofic-coinduction}, and choose a finitely generated $K\geq H$ with $[K:H]=\infty$. Since $K$ is virtually free, \cite[Theorem~1.1]{minasyan2026} gives $H\leq L\leq K$ with $[K:L]<\infty$ and a retraction $\rho:L\to H$. The equality $L=H$ would imply $[K:H]<\infty$, so $L\neq H$. Lemma~\ref{lem:sofic-coinduction} gives that $X$ is coinduced from $X\uphar L$ and $X\uphar L$ is coinduced from $X\uphar H$. Since $X$ is minimal, \cite[Lemma~7.1]{douchamelleraytsankov2026} implies that $X\uphar L$ is minimal.
		
		For $x\in X\uphar H$, define $\Delta(x)\in A^L$ by $\Delta(x)(\ell)=x(\rho(\ell))$. If $t\in L$, then the $tH$-slice of $\Delta(x)$ is given by $\Delta(x)_t(h)=x(\rho(t)h)=(\sigma_H(\rho(t)^{-1})x)(h)$, so coinduction gives $\Delta(x)\in X\uphar L$. Moreover, for $\ell_0\in L$, $\sigma_L(\ell_0)\Delta(x)=\Delta(\sigma_H(\rho(\ell_0))x)$. Thus $\Delta(X\uphar H)$ is a nonempty closed $L$-invariant subset of $X\uphar L$.
		
		Suppose that $X$ is not a singleton. Since $X$ is coinduced from $X\uphar H$, the subshift $X\uphar H$ is not a singleton. By $H$-invariance, choose $x_0,x_1\in X\uphar H$ with $x_0(1_H)\neq x_1(1_H)$. Since $\rho\uphar H=\operatorname{id}_H$ and $L\neq H$, the kernel of $\rho$ is nontrivial; fix $1_L\neq k\in\ker\rho$. Then $k\notin H$, and coinduction gives $y\in X\uphar L$ whose $H$-slice is $x_0$ and whose $kH$-slice is $x_1$. Hence $y(1_L)=x_0(1_H)\neq x_1(1_H)=y(k)$, whereas $\Delta(x)(1_L)=\Delta(x)(k)=x(1_H)$ for every $x\in X\uphar H$. Thus $\Delta(X\uphar H)$ is proper, contradicting minimality of $X\uphar L$. Every projectively isolated $G$-subshift is therefore a singleton. The basic neighborhood $\mathcal N_{\{0,1\}^G}^{\{1_G\}}$ contains no singleton subshift, so projectively isolated subshifts are not dense in $\Sub_G(\{0,1\})$. By \cite[Theorem~3.1]{doucha2024}, $G$ does not have \STRP.
		
		Consider now the complementary case. There is a finitely generated $H\leq G$ such that
		\begin{equation}\label{eq:lvf-trapped}
			[K:H]<\infty
			\
			\text{for every finitely generated }H\leq K\leq G.
		\end{equation}
		If $H$ is finite and $K_0\leq G$ is finitely generated, then $\langle H,K_0\rangle$ is finitely generated and \eqref{eq:lvf-trapped} gives $[\langle H,K_0\rangle:H]<\infty$. Thus $\langle H,K_0\rangle$, and hence $K_0$, is finite. The group $G$ is locally finite and therefore amenable. Since it is not finitely generated, \cite[Corollary~7.17]{douchamelleraytsankov2026} gives that $G$ does not have \STRP.
		
		We now use the special-symbol mechanism from
		\cite[Section~8.2]{douchamelleraytsankov2026}. Following \cite[Definition~8.5]{douchamelleraytsankov2026}, if $\Gamma$ is a group and $\Lambda\leq\Gamma$, let $\chi_\Lambda\in 2^\Gamma$ denote the characteristic function of $\Lambda$. The group $\Gamma$ has the \emph{special symbol property relative to $\Lambda$} if the subshift $\overline{\Gamma\cdot\chi_\Lambda}$ is sofic, where $\Gamma$ acts by the shift $\sigma_\Gamma$. The \emph{special symbol property} is the case $\Lambda=\{1_\Gamma\}$.
		
		Assume that $H$ is infinite. Since $H$ is finitely generated and $G$ is locally virtually free, $H$ is virtually free and hence hyperbolic. By \cite[Proposition~4.1]{dahmaniyaman2008}, $H$ has the special symbol property. Thus
		$S_H:=\overline{H\cdot\chi_{\{1_H\}}}
		=\{0^H\}\cup\{\chi_{\{h\}}:h\in H\}$
		is sofic. After enlarging the alphabet of an SFT cover if necessary, choose a finite alphabet $B$, an SFT $V\subseteq B^H$, and a surjection $q:B\to\{0,1\}$ such that $q^H(V)=S_H$. Choose $v_*\in V$ with $q^H(v_*)=\chi_{\{1_H\}}$, and let $p_*=v_*\uphar\{1_H\}\in B^{\{1_H\}}$. The use of a pattern in an SFT cover to detect the distinguished point of the
		special-symbol factor is the same mechanism used in the proof of
		\cite[Proposition~8.6]{douchamelleraytsankov2026}; here we use minimality,
		rather than recurrence, to obtain the contradiction.
		
		No minimal $H$-subshift $W\subseteq V$ satisfies $p_*\in W_{\{1_H\}}$. Suppose otherwise, and choose $w\in W$ with $w\uphar\{1_H\}=p_*$. Then $q^H(w)(1_H)=1$, so $q^H(w)=\chi_{\{1_H\}}$, since this is the unique point of $S_H$ taking the value $1$ at $1_H$. The set $q^H(W)$ is closed and $H$-invariant and contains $\chi_{\{1_H\}}$, so $S_H\subseteq q^H(W)$. Since $q^H(W)\subseteq S_H$, one has $q^H(W)=S_H$. But $q^H(W)$ is minimal, whereas $\{0^H\}$ is a proper nonempty subshift of $S_H$, a contradiction.
		
		Write $V=[E,\mathcal E]_H$ and let $V^\uparrow=[E,\mathcal E]_G$. Define
		$\mathcal U=\{Y\in\Sub_G(B):Y\subseteq V^\uparrow,\ p_*\in Y_{\{1_G\}}\}$.
		This set is nonempty: Lemma~\ref{lem:free-extension} gives $y\in V^\uparrow$ with $y\uphar H=v_*$, so $p_*\in(V^\uparrow)_{\{1_G\}}$. It is clopen. Indeed, $Y\subseteq V^\uparrow$ if and only if $Y_E\subseteq\mathcal E$, while $p_*\in Y_{\{1_G\}}$ depends only on the $\{1_G\}$-language; each condition is a finite union of basic clopen conditions.
		
		Suppose that $Y\in\mathcal U$ is projectively isolated. By \cite[Theorem~7.14]{douchamelleraytsankov2026}, $Y$ is minimal and sofic. Choose $H_0$ as in Lemma~\ref{lem:sofic-coinduction} and let $K=\langle H,H_0\rangle$. Then $K$ is finitely generated, so \eqref{eq:lvf-trapped} gives $[K:H]<\infty$. Lemma~\ref{lem:sofic-coinduction} gives that $Y$ is coinduced from $Y\uphar K$. By \cite[Lemma~7.1]{douchamelleraytsankov2026}, the subshift $Y\uphar K$ is minimal.
		
		Let $Z=Y\uphar K$ and $N=\bigcap_{k\in K}kHk^{-1}$, so $N\triangleleft K$. If $k_2=k_1h$ for some $h\in H$, then $k_2Hk_2^{-1}=k_1Hk_1^{-1}$. Thus the conjugate $kHk^{-1}$ depends only on the coset $kH$, and $[K:H]<\infty$ gives finitely many distinct conjugates, say $H_i=k_iHk_i^{-1}$ for $i<m$. Then $N=\bigcap_{i<m}H_i$. The map $K/N\to\prod_{i<m}K/H_i$ given by $gN\mapsto(gH_i)_{i<m}$ is injective, so $[K:N]\leq\prod_{i<m}[K:H_i]=[K:H]^m<\infty$.
		
		Choose an $N$-minimal (i.e. nonempty closed $ N $-invariant set without such proper subsets) subflow $C\subseteq Z$, and for $k\in K$ write $C_k=\sigma_K(k)C$. Since $N\triangleleft K$, each $C_k$ is $N$-minimal. Moreover, $k_1N=k_2N$ implies $C_{k_1}=C_{k_2}$. Thus $\{C_k:k\in K\}$ is finite. Its union is a nonempty closed $K$-invariant subset of the minimal $K$-flow $Z$, and therefore
		$Z=\bigcup_{k\in K}C_k$.
		
		For $k\in K$, let $D_k=\bigcup_{h\in H}C_{hk}$. Since the family $\{C_t:t\in K\}$ is finite, $D_k$ is a finite union of closed sets, and it is $H$-invariant. We claim that $D_k$ is $H$-minimal. Let $D\subseteq D_k$ be nonempty, closed, and $H$-invariant. Then $D\cap C_{hk}\neq\emptyset$ for some $h\in H$. Both sets are $N$-invariant, so $N$-minimality of $C_{hk}$ gives $C_{hk}\subseteq D$. For every $h'\in H$, $H$-invariance of $D$ gives
		$\sigma_K(h'h^{-1})C_{hk}=C_{h'k}\subseteq D$.
		Thus $D_k\subseteq D$, and hence $D=D_k$. Since $C_k\subseteq D_k$ for every $k\in K$, one has $Z=\bigcup_{k\in K}D_k$; only finitely many distinct $D_k$ occur. Thus $Z$ is a finite union of $H$-minimal subflows.
		
		Since $p_*\in Y_{\{1_G\}}$, choose $y\in Y$ with $y\uphar\{1_G\}=p_*$ and let $z=y\uphar K\in Z$. Choose $k\in K$ such that $z\in D_k$, and let $W=D_k\uphar H$. Restriction to $H$ is continuous and $H$-equivariant, so $H$-minimality of $D_k$ gives that $W$ is a minimal $H$-subshift. If $w\in W$, choose $x\in D_k$ with $x\uphar H=w$ and $\widetilde y\in Y$ with $\widetilde y\uphar K=x$. Since $Y\subseteq V^\uparrow$, Lemma~\ref{lem:free-extension} gives $\widetilde y\uphar H\in V$, so $w=x\uphar H=\widetilde y\uphar H\in V$. Thus $W\subseteq V$. Moreover, $z\in D_k$ and $z\uphar\{1_H\}=p_*$ give $p_*\in W_{\{1_H\}}$, contradicting the fact that no minimal $ H $-subshift of $ V $ contains $p_*$. Thus $\mathcal U$ contains no projectively isolated subshift, and \cite[Theorem~3.1]{doucha2024} gives that $G$ does not have \STRP.
	\end{proof}

	\section{Descriptive complexity}
	
	\subsection{Coding countable groups and finite symbolic data}
	
	We use the finite-level Borel hierarchy on Polish spaces. The classes $\bfsgm^0_1$ and $\bfpi^0_1$ are the open and closed sets. Recursively, $\bfsgm^0_{n+1}$ consists of countable unions of $\bfpi^0_n$ sets, $\bfpi^0_{n+1}$ consists of countable intersections of $\bfsgm^0_n$ sets, and $\bfdlt^0_n=\bfsgm^0_n\cap\bfpi^0_n$. A class $\mathcal A$ is \textbf{$\bfsgm^0_n$-hard} if every $\bfsgm^0_n$ subset of every zero-dimensional Polish space continuously reduces to $\mathcal A$; equivalently, it is enough to reduce one fixed $\bfsgm^0_n$-complete set. These definitions and the strictness of the Borel hierarchy are standard; see \cite{kechris1995}. All hardness statements below refer to continuous Wadge reductions.
	
	Let $F_\omega=\langle a_0,a_1,\ldots\rangle$ be the free group of countable rank. The product space $2^{F_\omega}$ is compact. Let
	$\mathrm{ctblgrp}=\NSub(F_\omega)$ be the set of normal subgroups of $F_\omega$. The subgroup and normality axioms are closed coordinate conditions; $\mathrm{ctblgrp}$ is a compact Polish subspace of $2^{F_\omega}$. For $N\in\mathrm{ctblgrp}$, write $G_N=F_\omega/N$ and
	\[
	\mathsf{STRP}=\{N\in\mathrm{ctblgrp}:G_N\text{ has \STRP}\}.
	\]
	
	Finite alphabets will be initial segments $A=\{0,\ldots,n-1\}$, $1\leq n<\omega$. A \textbf{finite symbolic code} is a triple $c=(A,\bar e,\mathcal E)$, where $\bar e=(e_0,\ldots,e_{m-1})\in F_\omega^m$ and $\mathcal E\subseteq A^m$. For $N\in\mathrm{ctblgrp}$, define
	\[
	\begin{aligned}
		Y_c(N) = \{x\in A^{F_\omega}: \forall u,v\in F_\omega \ &[u^{-1}v\in N \rightarrow x(u)=x(v)] \\
		\wedge \ \forall w\in F_\omega \ &[(x(w^{-1}e_i))_{i<m}\in\mathcal E]\}.
	\end{aligned}
	\]
	The first condition says that $x$ descends to a configuration on $G_N$; the second imposes the translated local rules coded by $(\bar e,\mathcal E)$.
	
	When nonempty, the set $Y_c(N)$ is the pullback to $F_\omega$ of the corresponding $G_N$-SFT. Let $\pi_N:F_\omega\to G_N=F_\omega/N$ be the quotient map. The pullback map $\pi_N^*:A^{G_N}\to A^{F_\omega},\ \pi_N^*(x)(w)=x(wN) $ is a homeomorphism onto
	\[
	\{y\in A^{F_\omega}: \forall u,v\in F_\omega \ [u^{-1}v\in N\rightarrow y(u)=y(v)]\}.
	\]
	It is equivariant in the sense that
	\[
	\forall w\in F_\omega \ [\pi_N^*(\sigma_{G_N}(wN)x)
	=
	\sigma_{F_\omega}(w)\pi_N^*(x)].
	\]
	Thus $Y_c(N)$ is the pullback subshift under $\pi_N^*$ of the $G_N$-SFT coded by $c$. In particular, when we write
	$Z\subseteq Y_c(N)$ and call $Z$ a subshift, we mean a nonempty closed
	$F_\omega$-invariant subset of $Y_c(N)$. Since every point of $Y_c(N)$ is
	constant on the cosets of $N$, such a $Z$ is equivalently the pullback of a
	unique $G_N$-subshift in $A^{G_N}$. We use these two descriptions
	interchangeably.
	
	For a finite tuple $\bar k=(k_0,\ldots,k_{\ell-1})\in F_\omega^\ell$, let
	$Y_c(N)_{\bar k}=\{(x(k_j))_{j<\ell}:x\in Y_c(N)\}$. Coordinates of $\bar e$ or $\bar k$ may represent the same element of $G_N$. Let $E_N=\{e_iN:i<m\}$ and define
	$\iota_N:A^{E_N}\to A^m$ by $\iota_N(r)(i)=r(e_iN)$. Its image consists exactly of the tuple-patterns compatible with the equalities modulo $N$.
	
	\begin{defn}\label{def:code-gr}
		The code $c=(A,\bar e,\mathcal E)$ is \textbf{globally realizable} over $G_N$ if
		$Y_c(N)_{\bar e}=\mathcal E\neq\emptyset$.
	\end{defn}
	
	When this holds, there is a unique $\mathcal E_N\subseteq A^{E_N}$ with $\iota_N(\mathcal E_N)=\mathcal E$, and $(A,E_N,\mathcal E_N)$ is globally realizable in the sense of Definition~\ref{def:globally-realizable}. Conversely, every globally realizable triple over $G_N$ is represented by a finite symbolic code after choosing word representatives in $F_\omega$.
	
	For a code $c$, a finite tuple $\bar k$, and $q\in A^{|\bar k|}$, define the \textbf{occurrence set}
	\[
	\operatorname{Occ}(c,\bar k,q)=\{N\in\mathrm{ctblgrp}:q\in Y_c(N)_{\bar k}\}.
	\]
	
	\begin{lem}\label{lem:occurrence-closed}
		Every set $\operatorname{Occ}(c,\bar k,q)$ is closed in $\mathrm{ctblgrp}$.
	\end{lem}
	
	\begin{proof}
		We have that 
		\[
		\operatorname{Occ}(c,\bar k,q)=p_2[ \{(N,x)\in\mathrm{ctblgrp}\times A^{F_\omega}: x\in Y_c(N) \  \wedge \ x\uphar \bar{k} = q  \} ],
		\]
		where $ p_2 $ is the projection to the corresponding coordinate. As it's a projection of a closed subset of a compact space, it is compact and closed.
	\end{proof}
	
	For a code $c=(A,\bar e,\mathcal E)$, let $\operatorname{GR}(c)$ be the set of $N$ over which $c$ is globally realizable. The rule at the identity gives $Y_c(N)_{\bar e}\subseteq\mathcal E$, and
	\[
	\operatorname{GR}(c)=
	\begin{cases}
		\emptyset,&\mathcal E=\emptyset,\\[1mm]
		\displaystyle\bigcap_{q\in\mathcal E}\operatorname{Occ}(c,\bar e,q),&\mathcal E\neq\emptyset.
	\end{cases}
	\]
	The set $\operatorname{GR}(c)$ is closed.
	
	\subsection{The universal-image condition}
	
	Fix codes $c=(A,\bar e,\mathcal E)$ and $d=(B,\bar f,\mathcal F)$, and a map $p:B\to A$. We also write $p$ for the induced one-block map. Given a finite tuple $\bar k=(k_0,\ldots,k_{\ell-1})$ and $q\in A^\ell$, define a code $d^{\neg(\bar k,q)}$ as follows. If $|\bar f|=m$, its window is $\bar f\,{}^\frown\bar k$, and its allowed family is
	\begin{equation}\label{eq:s3-coding-04}
		\mathcal F^{\neg q}
		=\{b\in B^{m+\ell}:(b_i)_{i<m}\in\mathcal F
		\text{ and }(p(b_{m+j}))_{j<\ell}\neq q\}.
	\end{equation}
	For $N\in\mathrm{ctblgrp}$, write $Y=Y_d(N)$ and $Y^{\neg(\bar k,q)}=Y_{d^{\neg(\bar k,q)}}(N)$. When nonempty, the latter is the subshift obtained from $Y$ by forbidding $q$ at every translate in the $p$-image.
	
	\begin{lem}\label{lem:finite-obstruction}
		Assume $N\in\operatorname{GR}(d)$. The following are equivalent.
		\begin{enumerate}[label=(\arabic*),font=\upshape]
			\item Every subshift $Z\subseteq Y$ with $Z_{\bar f}=\mathcal F$ satisfies $p(Z)=p(Y)$.
			\item For every finite tuple $\bar k$ and every $q\in p(Y)_{\bar k}$,
			$(Y^{\neg(\bar k,q)})_{\bar f}\neq\mathcal F$.
		\end{enumerate}
	\end{lem}
	
	\begin{proof}
		Suppose (1) fails. Choose a subshift $Z\subseteq Y$ with $Z_{\bar f}=\mathcal F$ and $p(Z)\subsetneqq p(Y)$, and fix $x\in p(Y)\setminus p(Z)$. Since $p(Z)$ is compact, there are a finite tuple $\bar k$ and $q=x\uphar \bar{k}$ such that the cylinder
		$\mathcal C(\bar k,q)=\{u\in A^{F_\omega}:u\uphar \bar{k}=q\}$
		satisfies $\mathcal C(\bar k,q)\cap p(Z)=\emptyset$. Thus $q\in p(Y)_{\bar k}$.
		
		For $z\in Z$ and $w\in F_\omega$, shift-invariance gives $\sigma_{F_\omega}(w)p(z)\in p(Z)$, so $ p(z)\uphar (w^{-1} \bar{k})\neq q $.
		Together with $Z\subseteq Y$, this is exactly the defining condition for $Z\subseteq Y^{\neg(\bar k,q)}$. Since $Y^{\neg(\bar k,q)}\subseteq Y$ by the definition of $\mathcal F^{\neg q}$, one has
		$\mathcal F=Z_{\bar f}\subseteq(Y^{\neg(\bar k,q)})_{\bar f}\subseteq Y_{\bar f}=\mathcal F$,
		contradicting (2).
		
		Suppose (2) fails. Then for some finite tuple $\bar k$ and $q\in p(Y)_{\bar k}$ one has $(Y^{\neg(\bar k,q)})_{\bar f}=\mathcal F$. Since $N\in\operatorname{GR}(d)$, $\mathcal F=Y_{\bar f}\neq\emptyset$, so $Z:=Y^{\neg(\bar k,q)}$ is nonempty. By the definition of $\mathcal F^{\neg q}$, $Z\subseteq Y$ and $q\notin p(Z)_{\bar k}$, while $q\in p(Y)_{\bar k}$. Thus $Z_{\bar f}=\mathcal F$ and $p(Z)\neq p(Y)$, contradicting (1).
	\end{proof}
	
	For $q\in A^{|\bar k|}$, let
	$\operatorname{Occ}_p(d,\bar k,q)=\{N:q\in p(Y_d(N))_{\bar k}\}$. The proof of Lemma~\ref{lem:occurrence-closed}, with an additional finite disjunction over the $p$-preimages of $q$, shows that this set is closed.
	
	Define
	\[
	\begin{aligned}
		\operatorname{Lang}(c,d,p)={}
		\bigcap_{q\in\mathcal E}\operatorname{Occ}_p(d,\bar e,q)
		\cap\bigcap_{q\in A^{|\bar e|}\setminus\mathcal E}
		\bigl(\mathrm{ctblgrp}\setminus\operatorname{Occ}_p(d,\bar e,q)\bigr).
	\end{aligned}
	\]
	This is a finite Boolean combination of closed sets. It follows that
	$\operatorname{Lang}(c,d,p)\in\bfdlt^0_2$.
	
	For a finite pair $(\bar k,q)$, define
	\[
	\begin{aligned}
		\operatorname{Bad}(d,p,\bar k,q)=
		\operatorname{GR}(d)\cap\operatorname{Occ}_p(d,\bar k,q)
		\cap\bigcap_{a\in\mathcal F}
		\operatorname{Occ}(d^{\neg(\bar k,q)},\bar f,a).
	\end{aligned}
	\]
	This set is closed. On $\operatorname{GR}(d)$, Lemma~\ref{lem:finite-obstruction} gives
	\[
	\operatorname{Univ}(d,p)=
	\operatorname{GR}(d)\cap
	\bigcap_{\bar k,q}
	\bigl(\mathrm{ctblgrp}\setminus\operatorname{Bad}(d,p,\bar k,q)\bigr)
	\in\bfpi^0_2.
	\]
	Finally, let
	\[
	\operatorname{Wit}(c,d,p)
	=\operatorname{Lang}(c,d,p)\cap\operatorname{Univ}(d,p).
	\]
	Then $\operatorname{Wit}(c,d,p)\in\bfpi^0_2$.
	
	\begin{thm}\label{thm:pi04}
		The collection $\mathsf{STRP}$ belongs to $\bfpi^0_4$.
	\end{thm}
	
	\begin{proof}
		Theorem~\ref{thm:finite-characterization} gives
		\begin{equation}\label{eq:s3-coding-09}
			\mathsf{STRP}=
			\bigcap_c\left[
			\bigl(\mathrm{ctblgrp}\setminus\operatorname{GR}(c)\bigr)
			\cup\bigcup_{d,p}\operatorname{Wit}(c,d,p)
			\right].
		\end{equation}
		Here $c$ ranges over all finite codes whose alphabet has at least two symbols, $d$ ranges over all finite codes, and $p$ ranges over the surjections from the alphabet of $d$ onto the alphabet of $c$. These are countable ranges.
		
		For fixed $c$, the complement of $\operatorname{GR}(c)$ is open, while the union over $(d,p)$ in \eqref{eq:s3-coding-09} is $\bfsgm^0_3$. The bracket is $\bfsgm^0_3$. Its countable intersection over $c$ is $\bfpi^0_4$.
	\end{proof}

	\subsection{A continuous free-rank reduction}
	
	Let $\mathrm{Grp}_\omega$ be the space of group structures with domain $\omega$. More explicitly,
	\begin{equation}\label{eq:s3-rank-01}
		\mathrm{Grp}_\omega=
		\{(\cdot,\operatorname{inv},e)\in
		\omega^{\omega\times\omega}\times\omega^{\omega}\times\omega:
		(\omega,\cdot,\operatorname{inv},e)\text{ is a group}\}.
	\end{equation}
	It carries the subspace topology inherited from the displayed product of
	discrete spaces. The group axioms define a closed subspace, so
	$\mathrm{Grp}_\omega$ is Polish. A basic open neighborhood prescribes
	finitely many values of multiplication, inverse, and the identity.
	
	Equip $2^{\omega\times\omega}$ with the product topology. For
	$x\in2^{\omega\times\omega}$ and $m\in\omega$, write
	$Z_m(x)=\{n\in\omega:x(m,n)=0\}$. Consider
	\begin{equation}\label{eq:s3-rank-02}
		\mathsf S_3=
		\{x\in2^{\omega\times\omega}:\exists m\,[Z_m(x)\text{ is infinite}]\}
		=
		\{x:\exists m\ \forall q\ \exists n\geq q\ [x(m,n)=0]\}.
	\end{equation}
	The set $\mathsf S_3$ is $\bfsgm^0_3$-complete for continuous reductions;
	see \cite[Section~23]{kechris1995}.
	
	Carson et al. use \emph{potential generators} to prove the effective
	$\Sigma^0_3$-completeness of finite generation among free groups
	\cite[proof of Proposition~2.4]{carson2012}; see also
	\cite[Proposition~3.1]{carson2012}. A potential generator is a basis symbol
	which remains free until a later stage identifies it with a word in previously
	retained generators. We call that identification a \emph{collapse}. The
	construction below is a continuous version, parametrized directly by
	$x\in2^{\omega\times\omega}$. We include the finite-stage labeling because the
	index-set result in \cite{carson2012} does not by itself give a continuous map
	into $\mathrm{Grp}_\omega$.
	
	If $\mathcal B$ is a set, $F(\mathcal B)$ denotes the free group with
	\textbf{free basis} $\mathcal B$: every nonidentity element has a unique
	reduced word over $\mathcal B\cup\mathcal B^{-1}$. For $t\in\mathcal B$,
	\begin{equation}\label{eq:s3-rank-03}
		F(\mathcal B)=F(\mathcal B\setminus\{t\})*\langle t\rangle.
	\end{equation}
	Every element outside $F(\mathcal B\setminus\{t\})$ has a unique reduced
	free-product normal form
	\begin{equation}\label{eq:s3-rank-04}
		u_0t^{m_1}u_1\cdots t^{m_r}u_r,
		\qquad
		r\geq1,\quad m_i\neq0,\quad
		u_i\in F(\mathcal B\setminus\{t\}),
	\end{equation}
	where $u_i\neq1$ for $0<i<r$. These are the usual reduced-word and
	free-product normal forms; see
	\cite[Chapter~I, Section~1 and Chapter~IV, Section~1]{lyndonschupp1977}.
	
	A collapse must preserve the inequalities already assigned to finitely many
	labeled elements. The next lemma supplies such a collapse by an explicit long
	word. It is the finite-diagram safeguard underlying the potential-generator
	argument of \cite[proof of Proposition~2.4]{carson2012}.
	
	\begin{lem}\label{lem:finite-collapse}
		Let $\mathcal B$ be a free basis containing distinct elements $a,b,t$,
		and let $\mathcal D\subseteq F(\mathcal B)\setminus\{1\}$ be finite.
		There is $N\geq1$ such that the homomorphism
		\begin{equation}\label{eq:s3-rank-05}
			\rho_N:F(\mathcal B)\longrightarrow F(\mathcal B\setminus\{t\}),
			\qquad
			\rho_N(t)=a^Nba^N,
			\qquad
			\rho_N(c)=c\quad(c\in\mathcal B\setminus\{t\}),
		\end{equation}
		satisfies $\rho_N(w)\neq1$ for every $w\in\mathcal D$.
	\end{lem}

	\begin{proof}
		Write $K=F(\mathcal B\setminus\{t\})$. Every element of
		$\mathcal D\cap K$ is fixed by $\rho_N$, so it is enough to consider
		$w\in\mathcal D\setminus K$.
		
		Since
		$F(\mathcal B)=K*\langle t\rangle$, the normal form theorem for free
		products gives a unique expression
		\begin{equation}\label{eq:s3-rank-06}
			w=u_0t^{m_1}u_1\cdots t^{m_r}u_r
		\end{equation}
		with the following properties: $r\geq1$, each
		$m_i\in\bbz\setminus\{0\}$, $u_i\in K$ for $0\leq i\leq r$, 
		$u_i\neq1$ for $0<i<r$, if any, and the endpoint terms $u_0,u_r$ are allowed
		to be $1$.
		
		For $u\in K$, let $|u|$ denote the length of its reduced word over
		$(\mathcal B\setminus\{t\})\cup
		(\mathcal B\setminus\{t\})^{-1}$. Only finitely many $ u_i $'s occur in the normal forms of the finitely many elements of
		$\mathcal D\setminus K$. Define
		\begin{equation}\label{eq:s3-rank-07}
			L=
			1+\max\bigl(\{0\}\cup
			\{|u_i|:u_i\text{ occurs in one of these normal forms}\}\bigr).
		\end{equation}
		Thus $|u_i|<L$ for every syllable which occurs.
		
		For $m\neq0$, let $\varepsilon=\operatorname{sgn}(m)$. Since
		$\rho_N(t)=a^Nba^N$, direct multiplication gives
		\begin{equation}\label{eq:s3-rank-08}
			\rho_N(t^m)=
			a^{\varepsilon N}b^\varepsilon
			\bigl(a^{2\varepsilon N}b^\varepsilon\bigr)^{|m|-1}
			a^{\varepsilon N}.
		\end{equation}
		Fix $N>L$.
		
		We first examine cancellation across a $u_i$, where $0<i<r$. Let
		$\varepsilon=\operatorname{sgn}(m_i)$ and
		$\delta=\operatorname{sgn}(m_{i+1})$. Formula \eqref{eq:s3-rank-08} shows that
		$\rho_N(t^{m_i})$ ends with
		$b^\varepsilon a^{\varepsilon N}$ and
		$\rho_N(t^{m_{i+1}})$ begins with
		$a^{\delta N}b^\delta$. Thus the only possible cancellation across
		$u_i$ occurs in
		$b^\varepsilon a^{\varepsilon N}u_i
		a^{\delta N}b^\delta$.
		
		Assume first that $u_i\notin\langle a\rangle$. We have $u_i=a^pva^q$, where $p,q\in\bbz$, the reduced word $v$ is nonempty,
		and its first and last letters are different from $a^{\pm1}$. Since $|u_i|<N$, we have
		$|p|<N$ and $|q|<N$. After free reduction the displayed portion
		becomes
		$b^\varepsilon a^{\varepsilon N+p}v
		a^{q+\delta N}b^\delta$.
		The two powers of $a$ are nontrivial, because
		$|\varepsilon N+p|\geq N-|p|>0$ and
		$|q+\delta N|\geq N-|q|>0$. Since the first and last letters of
		$v$ are not $a^{\pm1}$, no further cancellation occurs there.
		In particular, the two displayed $b$-letters cannot meet.
		
		Assume now that $u_i\in\langle a\rangle$. Since it's nonidentity, $u_i=a^q$ for some $q\neq0$. We have
		$|q|=|u_i|<N$. The portion across $u_i$ reduces to
		$b^\varepsilon a^{\varepsilon N+q+\delta N}b^\delta$.
		If $\varepsilon=-\delta$, the middle exponent equals $q$ and is
		nonzero. If $\varepsilon=\delta$, then
		$|\varepsilon N+q+\delta N|\geq2N-|q|>0$.
		Thus the two $b$-letters cannot meet in this case either.
		
		We next examine the left endpoint. Write
		$u_0=v_0a^p$, where $a^p$ is the maximal terminal power of $a$ in
		the reduced word for $u_0$; when $u_0=1$, take $v_0=1$ and $p=0$.
		Let $\varepsilon=\operatorname{sgn}(m_1)$. The product
		$u_0a^{\varepsilon N}$ reduces to
		$v_0a^{p+\varepsilon N}$. Since
		$|p|\leq|u_0|<N$, one has
		$|p+\varepsilon N|\geq N-|p|>0$. By maximality of $a^p$, no further
		cancellation occurs between $v_0$ and this nontrivial power of $a$.
		Consequently cancellation from the left cannot reach the first
		$b^\varepsilon$ in $\rho_N(t^{m_1})$. The right endpoint is analogous.
		
		If $r=1$, the two endpoint
		calculations already show that the $b^{\pm1}$-letters occurring in
		$\rho_N(t^{m_1})$ survive. If $r>1$, the preceding internal
		calculations show that cancellation across each $u_i$ cannot join
		the $b$-letters belonging to the values of adjacent $t$-powers, i.e. no
		$b^{\pm1}$-letter contributed by any $t^{m_i}$ is cancelled in the
		reduced form of $\rho_N(w)$. Hence $\rho_N(w)$ contains at least one
		$b^{\pm1}$-letter after free reduction, so $\rho_N(w)\neq1$.
		
		The integer $N>L$ was chosen from the uniform bound \eqref{eq:s3-rank-07}, and the
		same $N$ works simultaneously for every $w\in\mathcal D$.
	\end{proof}

	Fix the countable alphabet
	$\mathcal A=\{a,b\}\cup\{g_{k,j}:k,j\in\omega\}$ and an ordering of
	$\mathcal A$. Let $\operatorname{Red}(\mathcal A)$ be the set of finite
	reduced words over $\mathcal A\cup\mathcal A^{-1}$, and fix a bijection
	$v:\omega\to\operatorname{Red}(\mathcal A)$, written $v(i)=v_i$.
	
	Fix a bijection $\langle\cdot,\cdot\rangle:\omega^2\to\omega$ such that
	$\langle p,q\rangle\geq p$ for all $p,q\in\omega$. Define
	$\langle k,m,n\rangle=\langle k,\langle m,n\rangle\rangle$. Thus
	$(k,m,n)\mapsto\langle k,m,n\rangle$ is a bijection
	$\omega^3\to\omega$ and satisfies $\langle k,m,n\rangle\geq k$.
	
	For $x\in2^{\omega\times\omega}$, define
	$c_x:\omega\times\omega\to\omega$ by
	\begin{equation}\label{eq:s3-rank-09}
		c_x(k,s)=
		\bigl|\{(m,n)\in\omega^2:
		m\leq k,\ \langle k,m,n\rangle<s,\ x(m,n)=0\}\bigr|.
	\end{equation}
	Thus $c_x(k,s)$ counts the collapses of the $k$-th potential generator
	that occur before stage $s$. The inequality
	$\langle k,m,n\rangle\geq k$ ensures that, at a stage
	$s=\langle k,m,n\rangle$, the $k$-th potential generator has already
	been introduced.
	
	\begin{lem}\label{lem:carson-continuous}
		There is a continuous map
		$R:2^{\omega\times\omega}\to\mathrm{Grp}_\omega$ such that
		\begin{equation}\label{eq:s3-rank-10}
			R(x)\cong
			\begin{cases}
				F_{m(x)+2},&x\in\mathsf S_3,\text{ where }
				m(x)=\min\{m:Z_m(x)\text{ is infinite}\},\\
				F_\omega,&x\notin\mathsf S_3.
			\end{cases}
		\end{equation}
	\end{lem}
	
	\begin{proof}
		For $k\in\omega$, let
		\begin{equation}\label{eq:s3-rank-21}
			\mathcal I_k(x)=
			\{\langle k,m,n\rangle:m\leq k,\ x(m,n)=0\}.
		\end{equation}
		By the definition of $c_x$,
		$c_x(k,s)=|\mathcal I_k(x)\cap s |$. Since only the
		finitely many rows $0,\ldots,k$ occur in \eqref{eq:s3-rank-21},
		\begin{equation}\label{eq:s3-rank-22}
			\mathcal I_k(x)\text{ is finite}
			\ \longleftrightarrow\ 
			Z_m(x)\text{ is finite for every }m\leq k.
		\end{equation}
		
		Define
		\begin{equation}\label{eq:s3-rank-23}
			\mathcal S_x=
			\{a,b\}\cup
			\{g_{k,|\mathcal I_k(x)|}:
			\mathcal I_k(x)\text{ is finite}\}.
		\end{equation}
		Suppose $x\in\mathsf S_3$ and put
		\begin{equation}\label{eq:s3-rank-24}
			m(x)=\min\{m\in\omega:Z_m(x)\text{ is infinite}\}.
		\end{equation}
		By \eqref{eq:s3-rank-22}, $\mathcal I_k(x)$ is finite exactly when
		$k<m(x)$. Hence $F(\mathcal S_x)\cong F_{m(x)+2}$. If
		$x\notin\mathsf S_3$, every $Z_m(x)$ is finite, so every
		$\mathcal I_k(x)$ is finite and
		$F(\mathcal S_x)\cong F_\omega$.
		
		The remaining task is to realize $F(\mathcal S_x)$ by a group
		structure on $\omega$ which depends continuously on $x$. The set
		$\mathcal S_x$ itself is determined by the entire parameter $x$ and
		therefore cannot be used directly for this purpose. We approximate its
		basis symbols at finite stages.
		
		For $s\in\omega$, define
		\begin{equation}\label{eq:s3-rank-11}
			\mathcal B_s^x=
			\{a,b\}\cup\{g_{k,c_x(k,s)}:k\leq s\},
			\
			G_s^x=F(\mathcal B_s^x).
		\end{equation}
		Thus, for each $k\leq s$, the basis $\mathcal B_s^x$ contains one
		current symbol $g_{k,c_x(k,s)}$. The second index records the number
		of members of $\mathcal I_k(x)$ which have appeared before stage $s$.
		
		We first determine how this current symbol changes. Let
		$s=\langle k,m,n\rangle$. For every $\ell\leq s$, the count
		$c_x(\ell,s+1)$ differs from $c_x(\ell,s)$ only through pairs
		$(p,q)$ satisfying
		$\langle\ell,p,q\rangle=\langle k,m,n\rangle$. More explicitly,
		\[
		\begin{aligned}
			c_x(\ell,s+1)
			=\, &
			\bigl|\{(p,q):p\leq\ell,\
			\langle\ell,p,q\rangle<\langle k,m,n\rangle,\
			x(p,q)=0\}\bigr|\\
			+\, & \bigl|\{(p,q):p\leq\ell,\
			\langle\ell,p,q\rangle=\langle k,m,n\rangle,\
			x(p,q)=0\}\bigr|.
		\end{aligned}
		\]
		The first term is $c_x(\ell,s)$. The second term equals $1$ exactly when
		$\ell=k$, $m\leq k$, and $x(m,n)=0$. Hence
		$c_x(\ell,s+1)=c_x(\ell,s)+1$ exactly in this case, and otherwise
		$c_x(\ell,s+1)=c_x(\ell,s)$.
		
		Moreover, $c_x(s+1,s+1)=0$, since
		$\langle s+1,p,q\rangle\geq s+1$ for all $p,q\in\omega$. Put
		$c=c_x(k,s)$. It follows that
		\begin{equation}\label{eq:s3-rank-12}
			\mathcal B_{s+1}^x=
			\begin{cases}
				\mathcal B_s^x\cup\{g_{s+1,0}\},
				&m>k\text{ or }x(m,n)=1,\\
				(\mathcal B_s^x\setminus\{g_{k,c}\})
				\cup\{g_{k,c+1},g_{s+1,0}\},
				&m\leq k\text{ and }x(m,n)=0.
			\end{cases}
		\end{equation}
		
		The relation between the stage bases $\mathcal B_s^x$ and the target
		basis $\mathcal S_x$ can be read directly from $\mathcal I_k(x)$.
		Fix $k\in\omega$.
		
		Suppose first that
		$\mathcal I_k(x)=\{s_0<\cdots<s_{q-1}\}$ is finite. Since $ c_x(k,s)=|\mathcal I_k(x)\cap s | $,
		one has $c_x(k,s_j)=j$ for $0\leq j<q$. Hence, in the transition from
		$\mathcal B_{s_j}^x$ to $\mathcal B_{s_j+1}^x$, the symbol $g_{k,j}$
		is removed and replaced by $g_{k,j+1}$. After the last such stage,
		$c_x(k,s)=q$ for all sufficiently large $s$. Thus $g_{k,q}$ belongs
		to $\mathcal B_s^x$ for all sufficiently large $s$, while
		$g_{k,0},\ldots,g_{k,q-1}$ are successively removed.
		
		Suppose instead that $\mathcal I_k(x)$ is infinite, and write its
		increasing enumeration as $(s_j)_{j\in\omega}$. Then
		$c_x(k,s_j)=j$ for every $j\in\omega$. Hence $g_{k,j}$ is removed in the
		transition from $\mathcal B_{s_j}^x$ to $\mathcal B_{s_j+1}^x$ and
		replaced by $g_{k,j+1}$. Therefore no fixed symbol $g_{k,j}$ remains
		in all sufficiently late stage bases.
		
		Consequently,
		\[
		g_{k,j}\in\mathcal S_x
		\ \longleftrightarrow\ 
		(\exists s_0)(\forall s\geq s_0)\,
		g_{k,j}\in\mathcal B_s^x,
		\]
		and, when this occurs, necessarily
		$j=|\mathcal I_k(x)|$.
		
		In this precise sense, the bases $\mathcal B_s^x$ provide finite-stage
		approximations to the free basis $\mathcal S_x$: a current symbol is
		replaced whenever new evidence enters $\mathcal I_k(x)$, and it
		stabilizes exactly when $\mathcal I_k(x)$ is finite.
		
		We define by induction integers $n_s^x$ and injections
		$\lambda_s^x:\{0,\ldots,n_s^x-1\}\to G_s^x$. Set
		\begin{equation}\label{eq:s3-rank-13}
			n_0^x=3,\qquad
			\lambda_0^x(0)=1,\qquad
			\lambda_0^x(1)=a,\qquad
			\lambda_0^x(2)=b.
		\end{equation}
		Suppose $n_s^x$ and $\lambda_s^x$ have been defined. Let
		\begin{equation}\label{eq:s3-rank-14}
			\mathcal D_s^x=
			\{\lambda_s^x(i)\lambda_s^x(j)^{-1}:i<j<n_s^x\}.
		\end{equation}
		Injectivity of $\lambda_s^x$ gives
		$1\notin\mathcal D_s^x$.
		
		Write $s=\langle k,m,n\rangle$ and $c=c_x(k,s)$. If
		$m>k$ or $x(m,n)=1$, then
		$\mathcal B_s^x\subseteq\mathcal B_{s+1}^x$ by
		\eqref{eq:s3-rank-12}; let
		$\pi_s^x:G_s^x\to G_{s+1}^x$ be the homomorphism induced by this
		inclusion.
		
		Suppose $m\leq k$ and $x(m,n)=0$. For $N\geq1$, let
		$\rho_{s,N}:G_s^x\to G_{s+1}^x$ be the homomorphism which fixes
		every element of
		$\mathcal B_s^x\setminus\{g_{k,c}\}$ and sends
		$g_{k,c}$ to $a^Nba^N$. Its image is contained in the free subgroup
		of $G_{s+1}^x$ generated by
		$\mathcal B_s^x\setminus\{g_{k,c}\}$. By
		Lemma~\ref{lem:finite-collapse}, the set
		\[
		\{N\geq1:
		\rho_{s,N}(w)\neq1\text{ for every }w\in\mathcal D_s^x\}
		\]
		is nonempty. Let $N_s^x$ be its least member and put
		$\pi_s^x=\rho_{s,N_s^x}$.
		
		In either case, $\pi_s^x$ is injective on the range of
		$\lambda_s^x$. Indeed, if $i<j<n_s^x$, then
		$\lambda_s^x(i)\lambda_s^x(j)^{-1}\in\mathcal D_s^x$, whence
		\[
		\pi_s^x(\lambda_s^x(i))
		\neq
		\pi_s^x(\lambda_s^x(j)).
		\]
		
		Put
		\begin{equation}\label{eq:s3-rank-16}
			E_s^x=
			\pi_s^x\bigl(
			\lambda_s^x[\{0,\ldots,n_s^x-1\}]
			\bigr).
		\end{equation}
		Thus $|E_s^x|=n_s^x$. For a finite set $E$ in a group, write
		$E^{-1}=\{z^{-1}:z\in E\}$ and
		$EE=\{zw:z,w\in E\}$. Define
		\begin{equation}\label{eq:s3-rank-17}
			\begin{aligned}
				\mathcal C_s^x=\, &
				E_s^x\cup\{1\}\cup(E_s^x)^{-1}\cup E_s^xE_s^x\\
				\cup \, &
				\{v_i:i\leq s,\
				\operatorname{alph}(v_i)\subseteq\mathcal B_{s+1}^x\},
			\end{aligned}
		\end{equation}
		where each $v_i$ in the second line is identified with the element
		of $G_{s+1}^x$ represented by that reduced word.
		
		The set $\mathcal C_s^x$ is finite and contains $E_s^x$. Put
		$n_{s+1}^x=|\mathcal C_s^x|$. For $i<n_s^x$, define
		$\lambda_{s+1}^x(i)=\pi_s^x(\lambda_s^x(i))$. Every element of
		$\mathcal C_s^x\setminus E_s^x$ has a unique reduced word over
		$\mathcal B_{s+1}^x$, hence a unique index in the fixed enumeration
		$(v_i)_{i\in\omega}$. Order these finitely many elements by those
		indices and assign them to
		$n_s^x,\ldots,n_{s+1}^x-1$. Then $\lambda_{s+1}^x$ is injective and
		\begin{equation}\label{eq:s3-rank-18}
			\lambda_{s+1}^x\upharpoonright n_s^x
			=\pi_s^x\circ\lambda_s^x,
			\qquad
			\lambda_{s+1}^x[
			\{0,\ldots,n_{s+1}^x-1\}]
			=\mathcal C_s^x.
		\end{equation}
		
		For $s<t$, write
		$\pi_{s,t}^x=\pi_{t-1}^x\circ\cdots\circ\pi_s^x$, and put
		$\pi_{s,s}^x=\operatorname{id}_{G_s^x}$. Induction on $t-s$ using
		\eqref{eq:s3-rank-18} gives
		\[
		\lambda_t^x(i)
		=
		\pi_{s,t}^x(\lambda_s^x(i))
		\qquad
		(i<n_s^x).
		\]
		
		The sequence $(n_s^x)$ is unbounded. Indeed, each of the distinct
		reduced words $1,a,\ldots,a^M$ occurs in the fixed enumeration
		$(v_i)$. Choose $s$ at least as large as their finitely many
		indices $ i $. Since $a\in\mathcal B_{s+1}^x$, all these words belong to
		$\mathcal C_s^x$, so $n_{s+1}^x\geq M+1$. Therefore
		$\bigcup_s\{0,\ldots,n_s^x-1\}=\omega$.
		
		Define multiplication and inverse on $\omega$ by
		\begin{equation}\label{eq:s3-rank-19}
			\begin{aligned}
				i\cdot_x j=\ell
				&\longleftrightarrow
				\exists s\,
				[i,j,\ell<n_s^x\ \wedge\
				\lambda_s^x(i)\lambda_s^x(j)=\lambda_s^x(\ell)],\\
				\operatorname{inv}_x(i)=j
				&\longleftrightarrow
				\exists s\,
				[i,j<n_s^x\ \wedge\
				\lambda_s^x(i)^{-1}=\lambda_s^x(j)].
			\end{aligned}
		\end{equation}
		Use $0$ as the identity.
		
		These relations define total functions. Given $i,j$, choose $s$
		with $i,j<n_s^x$. The elements
		$\lambda_{s+1}^x(i)$ and $\lambda_{s+1}^x(j)$ belong to $E_s^x$.
		By \eqref{eq:s3-rank-17}, their product and the inverse of
		$\lambda_{s+1}^x(i)$ belong to $\mathcal C_s^x$, hence receive
		labels at stage $s+1$.
		
		The values are unique. Suppose, for example, that
		$i\cdot_x j=\ell$ is witnessed at stage $s$ and
		$i\cdot_x j=\ell'$ is witnessed at stage $t$. Choose
		$u\geq s,t$. Compatibility of the labelings gives
		\[
		\lambda_u^x(\ell)
		=
		\lambda_u^x(i)\lambda_u^x(j)
		=
		\lambda_u^x(\ell').
		\]
		Injectivity of $\lambda_u^x$ gives $\ell=\ell'$. The same argument
		applies to inverses.
		
		The operations satisfy the group axioms. For example, let
		$p=i\cdot_x j$, $q=j\cdot_x k$,
		$r=p\cdot_x k$, and $t=i\cdot_x q$. Choose one stage $s$ at which
		the four defining product equalities are all witnessed. Then
		\[
		\lambda_s^x(r)
		=
		(\lambda_s^x(i)\lambda_s^x(j))\lambda_s^x(k)
		=
		\lambda_s^x(i)(\lambda_s^x(j)\lambda_s^x(k))
		=
		\lambda_s^x(t),
		\]
		so $r=t$ by injectivity. This proves associativity. Since
		$\lambda_s^x(0)=1$ at every stage, $0$ is a two-sided identity;
		the inverse identities follow in the same way from
		\eqref{eq:s3-rank-19}. Denote the resulting group structure on
		$\omega$ by $R(x)$.
		
		We now identify $R(x)$ with $F(\mathcal S_x)$. The analysis above
		gives the following dichotomy for every
		$g_{k,j}\in\mathcal B_s^x$. If $g_{k,j}\in\mathcal S_x$, then
		$\mathcal I_k(x)$ is finite, $j=|\mathcal I_k(x)|$, and
		$g_{k,j}\in\mathcal B_u^x$ for every $u\geq s$. If
		$g_{k,j}\notin\mathcal S_x$, then there is a unique $r\geq s$ such
		that
		$g_{k,j}\in\mathcal B_r^x$ but
		$g_{k,j}\notin\mathcal B_{r+1}^x$. At this transition
		$g_{k,j}$ is replaced by $g_{k,j+1}$, and the definition of
		$\pi_r^x$ gives
		$\pi_r^x(g_{k,j})=a^{N_r^x}ba^{N_r^x}$.
		
		For each $s$, define a homomorphism
		$\theta_s:G_s^x\to F(\mathcal S_x)$ on the free basis
		$\mathcal B_s^x$ as follows. Set
		$\theta_s(a)=a$ and $\theta_s(b)=b$. If
		$g_{k,j}\in\mathcal B_s^x\cap\mathcal S_x$, set
		$\theta_s(g_{k,j})=g_{k,j}$. If
		$g_{k,j}\in\mathcal B_s^x\setminus\mathcal S_x$, let $r\geq s$
		be its unique removal stage described above and set
		$\theta_s(g_{k,j})=a^{N_r^x}ba^{N_r^x}$.
		
		Informally, $\theta_s$ records the eventual image of an element of
		$G_s^x$ under all subsequent transition maps. In this sense one may
		think of
		\[
		\theta_s
		\sim
		\cdots\circ\pi_{s+1}^x\circ\pi_s^x,
		\]
		after identifying the basis symbols which eventually stabilize with
		the corresponding elements of $\mathcal S_x$.
		
		We claim that $\theta_{s+1}\circ\pi_s^x=\theta_s$ for every $s$.
		It suffices to check the equality on the free basis $\mathcal B_s^x$. 
		
		Let $d\in\mathcal B_s^x$. If $\pi_s^x(d)=d$, then either
		$d\in\mathcal S_x$, in which case
		$\theta_{s+1}(d)=d=\theta_s(d)$, or
		$d$ is some $g_{k,j}\notin\mathcal S_x$. In the latter case $d$ has a unique
		removal stage $r>s$, and both $\theta_s(d)$ and $\theta_{s+1}(d)$ are
		equal to $a^{N_r^x}ba^{N_r^x}$.
		
		If $\pi_s^x(d)\neq d$, then $d=g_{k,c_x(k,s)}$ is the unique symbol
		removed at stage $s$. By definition,
		$\pi_s^x(d)=a^{N_s^x}ba^{N_s^x}$, and therefore
		\[
		\theta_{s+1}(\pi_s^x(d))
		=
		a^{N_s^x}ba^{N_s^x}
		=
		\theta_s(d).
		\]
		Thus $\theta_{s+1}\circ\pi_s^x=\theta_s$.
		
		For $i\in\omega$, choose $s$ with $i<n_s^x$ and define
		$\Theta_x(i)=\theta_s(\lambda_s^x(i))$. This does not depend on
		the choice of $s$. Indeed, if $s\leq t$ and $i<n_s^x$, then
		\[
		\theta_t(\lambda_t^x(i))
		=
		\theta_t(\pi_{s,t}^x(\lambda_s^x(i)))
		=
		\theta_s(\lambda_s^x(i)).
		\]
		
		The map $\Theta_x:R(x)\to F(\mathcal S_x)$ is a homomorphism.
		For example, suppose $i\cdot_x j=\ell$. By
		\eqref{eq:s3-rank-19}, choose $s$ such that
		$i,j,\ell<n_s^x$ and
		$\lambda_s^x(i)\lambda_s^x(j)=\lambda_s^x(\ell)$. Applying
		$\theta_s$ gives
		$\Theta_x(i)\Theta_x(j)=\Theta_x(\ell)$. The inverse relation is
		verified in the same way.

		We prove that $\Theta_x$ is injective. Suppose
		$\Theta_x(i)=\Theta_x(j)$, and choose $s$ with
		$i,j<n_s^x$. Put
		$d=\lambda_s^x(i)\lambda_s^x(j)^{-1}\in G_s^x$. Then
		$\theta_s(d)=1$.
		
		For $w\in G_r^x$, let
		$\operatorname{gsupp}_r(w)$ be the set of symbols $g_{k,j}$ occurring
		in the reduced word for $w$ over the free basis $\mathcal B_r^x$.
		
		For every $r$ and $w\in G_r^x$,
		\[
		\operatorname{gsupp}_{r+1}(\pi_r^x(w))
		\subseteq
		\operatorname{gsupp}_r(w).
		\]
		Indeed, by definition, $\pi_r^x$ either fixes every $g$-symbol, or replaces the unique
		symbol removed at stage $r$ by a word in $a^{\pm1},b^{\pm1}$; subsequent
		free reduction can only delete letters. Hence, for $s\leq t$,
		\[
		\operatorname{gsupp}_t(\pi_{s,t}^x(d))
		\subseteq
		\operatorname{gsupp}_s(d).
		\]
		
		The finite set
		$\operatorname{gsupp}_s(d)\setminus\mathcal S_x$ consists of symbols
		having unique removal stages. Choose $t\geq s$ strictly larger than all
		these stages. None of these symbols belongs to $\mathcal B_t^x$.
		Therefore every $g$-symbol occurring in the reduced word for
		$\pi_{s,t}^x(d)$ belongs to $\mathcal S_x$, and hence
		\[
		\pi_{s,t}^x(d)\in F(\mathcal B_t^x\cap\mathcal S_x).
		\]
		
		On $F(\mathcal B_t^x\cap\mathcal S_x)$, the map $\theta_t$ is the
		homomorphism induced by the inclusion
		$\mathcal B_t^x\cap\mathcal S_x\subseteq\mathcal S_x$, and is
		therefore injective. By compatibility,
		\[
		\theta_t(\pi_{s,t}^x(d))
		=
		\theta_s(d)
		=
		1.
		\]
		Thus $\pi_{s,t}^x(d)=1$. Using
		$\lambda_t^x=\pi_{s,t}^x\circ\lambda_s^x$ on the first $n_s^x$
		labels, we obtain
		$\lambda_t^x(i)=\lambda_t^x(j)$. Injectivity of $\lambda_t^x$
		gives $i=j$.
		
		We next prove that $\Theta_x$ is surjective. From
		\eqref{eq:s3-rank-13},
		$\Theta_x(1)=a$ and $\Theta_x(2)=b$. Let
		$g_{k,j}\in\mathcal S_x$. By the characterization of
		$\mathcal S_x$ established above, there is $s_0$ such that
		$g_{k,j}\in\mathcal B_s^x$ for every $s\geq s_0$. Choose
		$r\in\omega$ such that $g_{k,j}$ is
		$v_r$, and choose $s\geq\max\{s_0,r\}$. Then
		$\operatorname{alph}(v_r)=\{g_{k,j}\}\subseteq
		\mathcal B_{s+1}^x$. By \eqref{eq:s3-rank-17},
		$g_{k,j}=v_r\in\mathcal C_s^x$. Hence
		\eqref{eq:s3-rank-18} gives an $i<n_{s+1}^x$ such that
		$\lambda_{s+1}^x(i)=g_{k,j}$. Since
		$g_{k,j}\in\mathcal S_x$,
		$\Theta_x(i)=\theta_{s+1}(g_{k,j})=g_{k,j}$.
		
		Thus the range of $\Theta_x$ contains every member of the free
		basis $\mathcal S_x$, so $\Theta_x$ is surjective. Therefore
		$R(x)\cong F(\mathcal S_x)$. By the computation of
		$\mathcal S_x$ at the beginning of the proof,
		\begin{equation}\label{eq:s3-rank-25}
			R(x)\cong
			\begin{cases}
				F_{m(x)+2},&x\in\mathsf S_3,\\
				F_\omega,&x\notin\mathsf S_3.
			\end{cases}
		\end{equation}
		
		It remains to prove continuity. For $s\in\omega$, let
		\[
		Q_s=
		\{(m,n)\in\omega^2:
		(\exists k)\ \langle k,m,n\rangle<s\}.
		\]
		The set $Q_s$ is finite. Fix
		$s\in\omega$, and suppose that
		$x\upharpoonright Q_s=y\upharpoonright Q_s$. We prove by induction on $r\leq s$ that
		\[
		\mathcal B_r^x=\mathcal B_r^y,\qquad
		G_r^x=G_r^y,\qquad
		n_r^x=n_r^y,\qquad
		\lambda_r^x=\lambda_r^y.
		\]
		
		For $r=0$, one has $c_x(0,0)=c_y(0,0)=0$, so
		$\mathcal B_0^x=\mathcal B_0^y=\{a,b,g_{0,0}\}$ and
		$G_0^x=G_0^y$. Also $n_0^x=n_0^y=3$, and
		\eqref{eq:s3-rank-13} gives $\lambda_0^x=\lambda_0^y$.
		
		Now suppose $r<s$ and that the asserted equalities hold at stage $r$.
		We first prove $\pi_r^x=\pi_r^y  $. Write $r=\langle k,m,n\rangle$. Since $r<s$, one has
		$(m,n)\in Q_s$, so $x(m,n)=y(m,n)$. Moreover, for every
		$\ell\leq r$, every coordinate $(p,q)$ entering the definition of
		$c_x(\ell,r)$ satisfies
		$\langle\ell,p,q\rangle<r<s$, and hence $(p,q)\in Q_s$. Therefore
		$c_x(\ell,r)=c_y(\ell,r)$ for every $\ell\leq r$.
		
		It follows from \eqref{eq:s3-rank-12} that the same transition of stage
		bases occurs for $x$ and $y$, and in particular
		$\mathcal B_{r+1}^x=\mathcal B_{r+1}^y$. If no symbol is removed, then
		$\pi_r^x$ and $\pi_r^y$ are induced by the same inclusion. If a symbol
		is removed, then the same symbol is removed in both constructions.
		Since the induction hypothesis gives
		$\mathcal D_r^x=\mathcal D_r^y$, the same homomorphisms
		$\rho_{r,N}$ are tested against the same finite set. Hence
		$N_r^x=N_r^y$ and again $\pi_r^x=\pi_r^y$.
		
		Now $\pi_r^x=\pi_r^y$ and $\lambda_r^x=\lambda_r^y$ give
		$E_r^x=E_r^y$. Together with
		$\mathcal B_{r+1}^x=\mathcal B_{r+1}^y$, this gives
		$\mathcal C_r^x=\mathcal C_r^y$, hence
		$n_{r+1}^x=n_{r+1}^y$ and
		$\lambda_{r+1}^x=\lambda_{r+1}^y$. Also
		$G_{r+1}^x=G_{r+1}^y$. This completes the induction.
		
		Fix $i,j,\ell\in\omega$. For $s\in\omega$, let
		\[
		\mathscr M_{i,j}^{\ell}(s)
		=
		\left\{
		\eta\in2^{Q_s}:
		\begin{array}{l}
			\text{for some, equivalently every, }x\in2^{\omega\times\omega}
			\text{ with }x\upharpoonright Q_s=\eta,\\
			i,j,\ell<n_s^x
			\text{ and }
			\lambda_s^x(i)\lambda_s^x(j)=\lambda_s^x(\ell)
		\end{array}
		\right\}.
		\]
		This is well defined by the finite-stage dependence proved above. For
		$\eta\in2^{Q_s}$, write $ [\eta]
		=
		\{x\in2^{\omega\times\omega}:x\upharpoonright Q_s=\eta\} $. Define $ M_{i,j}^{\ell}=\{x\in2^{\omega\times\omega}:i\cdot_x j=\ell\} $. By \eqref{eq:s3-rank-19},
		\[
		M_{i,j}^{\ell}
		=
		\bigcup_{s\in\omega}
		\{x:i,j,\ell<n_s^x,\
		\lambda_s^x(i)\lambda_s^x(j)=\lambda_s^x(\ell)\}
		=
		\bigcup_{s\in\omega}
		\bigcup_{\eta\in\mathscr M_{i,j}^{\ell}(s)}
		[\eta].
		\]
		Thus $M_{i,j}^{\ell}$ is open. Since multiplication in $R(x)$ is
		total and single-valued, we have that $ 2^{\omega\times\omega}\setminus M_{i,j}^{\ell} = \bigcup_{q\in\omega\setminus\{\ell\}}M_{i,j}^{q} $. Hence the complement is open, and $M_{i,j}^{\ell}$ is clopen.
		
		Similarly, for $i,j,s\in\omega$, let
		\[
		\mathscr I_i^j(s)
		=
		\left\{
		\eta\in2^{Q_s}:
		\begin{array}{l}
			\text{for some, equivalently every, }x
			\text{ with }x\upharpoonright Q_s=\eta,\\
			i,j<n_s^x
			\text{ and }
			\lambda_s^x(i)^{-1}=\lambda_s^x(j)
		\end{array}
		\right\},
		\]
		and put $ I_i^j = \{x\in2^{\omega\times\omega}:\operatorname{inv}_x(i)=j\} $. Again by \eqref{eq:s3-rank-19}, $ I_i^j = \bigcup_{s\in\omega} \bigcup_{\eta\in\mathscr I_i^j(s)} [\eta] $. Thus $I_i^j$ is open. Since inversion is total and single-valued, $ 2^{\omega\times\omega}\setminus I_i^j = \bigcup_{q\in\omega\setminus\{j\}}I_i^q $. Hence $I_i^j$ is clopen.
		
		Finally, the identity of the group is always $ 0 $. Therefore every
		coordinate map
		$x\mapsto i\cdot_x j$, $x\mapsto\operatorname{inv}_x(i)$, and
		$x\mapsto e_{R(x)}$ is continuous, and $ R:2^{\omega\times\omega}\to\mathrm{Grp}_\omega $ is continuous.
	\end{proof}

	\begin{lem}\label{lem:group-to-nsub}
		There is a continuous map $ \kappa:\mathrm{Grp}_\omega\to\NSub(F_\omega) $ such that for every $G\in\mathrm{Grp}_\omega$ we have $F_\omega/\kappa(G)\cong G$.
	\end{lem}
	
	\begin{proof}
		For a group structure $G$ on $\omega$, define $ \rho_G:F_\omega\to G, \ \rho_G(a_n)=n,\ \kappa(G)=\ker(\rho_G) $. The map $\rho_G$ is onto, and the first isomorphism theorem gives $ F_\omega/\kappa(G)\cong G $. For a fixed word $w\in F_\omega$, membership
		$w\in\kappa(G)$ is decided by evaluating the finitely many
		multiplication and inverse operations occurring in $w$ and comparing
		the result with the identity of $G$. This is a clopen condition on
		$\mathrm{Grp}_\omega$.
	\end{proof}
	
	Let $\Theta=\kappa\circ R$. The preceding lemmas give
	$G_{\Theta(x)}\cong R(x)$ for every $x\in2^{\omega\times\omega}$.
	
	\begin{prop}\label{prop:sigma03-hard}
		The collection $\mathsf{STRP}$ is $\bfsgm^0_3$-hard.
	\end{prop}
	
	\begin{proof}
		Every finite-rank free group has \STRP by
		\cite[Theorem~1.1]{kwiatkowska2012}, while $F_\omega$ does not have
		\STRP by \cite[Section~6.1, p.~30]{kechrisrosendal2007}.
		Lemma~\ref{lem:carson-continuous} gives $ x\in\mathsf S_3 \ \longleftrightarrow\ \Theta(x)\in\mathsf{STRP} $. The map $\Theta$ is continuous and $\mathsf S_3$ is
		$\bfsgm^0_3$-complete.
	\end{proof}

\subsection{Finitely generated loci and the rank dichotomy}

Define
\[
\mathsf{FG}
=
\{N\in\mathrm{ctblgrp}:G_N\text{ is finitely generated}\},
\]
\[
\mathsf{FP}
=
\{N\in\mathrm{ctblgrp}:G_N\text{ is finitely presented}\}.
\]

The corresponding computable index-set complexity of finite generation is
known: Carson et al.\ prove a $\Sigma^0_3$-completeness result already
within the class of free groups \cite[Proposition~3.1]{carson2012}.
We record the continuous Borel version needed here, together with its
finite-presentation consequences.

\begin{prop}\label{prop:sigma03-loci}
	The following statements hold.
	\begin{enumerate}[label=(\arabic*),font=\upshape]
		\item For every finitely generated group $H$, the isomorphism class
		\[
		[H]_{\cong}
		:=
		\{N\in\mathrm{ctblgrp}:G_N\cong H\}
		\]
		belongs to $\bfsgm^0_3$.
		
		\item The sets $\mathsf{FG},\ \mathsf{FP},\ \mathsf{FP}\cap\mathsf{STRP} $ are $\bfsgm^0_3$-complete.
	\end{enumerate}
\end{prop}

\begin{proof}
	We first prove (1). Fix a finitely generated group $H$ and choose a generating tuple $h_0,\ldots,h_{k-1}$ for $H$. Let $F_k=\langle x_0,\ldots,x_{k-1}\rangle$ be the free group on $k$ generators. Let $\rho_H:F_k\to H$ be the epimorphism determined by $\rho_H(x_i)=h_i$, and write $K_H=\ker(\rho_H)$. By the first isomorphism theorem, $H\cong F_k/K_H$.

	For $\bar w=(w_0,\ldots,w_{k-1})\in F_\omega^k$ and $N\in\mathrm{ctblgrp}$, define $ \varphi_{N,\bar w}:F_k\to G_N$, $ \varphi_{N,\bar w}(x_i)=w_iN $.
	We first characterize those $N$ for which $G_N\cong H$:
	\begin{equation}\label{eq:s3-rank-iso-characterization}
		G_N\cong H
		\ \longleftrightarrow\
		\exists\bar w\in F_\omega^k\ 
		\bigl[
		\varphi_{N,\bar w}\text{ is onto and }
		\ker(\varphi_{N,\bar w})=K_H
		\bigr].
	\end{equation}
	
	Assume first that $\theta:H\to G_N$ is an isomorphism. The composition $\theta\circ\rho_H:F_k\to G_N$ is surjective. For every $i<k$, choose $w_i\in F_\omega$ with $w_iN=\theta(h_i)$. For $\bar w=(w_0,\ldots,w_{k-1})$, the maps $\varphi_{N,\bar w}$ and $\theta\circ\rho_H$ agree on the free generators $x_0,\ldots,x_{k-1}$. The universal property of $F_k$ gives $\varphi_{N,\bar w}=\theta\circ\rho_H$. In particular, $\varphi_{N,\bar w}$ is onto. Since $\theta$ is injective, $\ker(\varphi_{N,\bar w})=\ker(\theta\circ\rho_H)=\ker(\rho_H)=K_H$.
	
	Conversely, suppose that $\bar w\in F_\omega^k$ satisfies that $\varphi_{N,\bar w}$ is onto and $\ker(\varphi_{N,\bar w})=K_H$. The first isomorphism theorem gives $G_N\cong F_k/\ker(\varphi_{N,\bar w})=F_k/K_H\cong H$. This proves \eqref{eq:s3-rank-iso-characterization}.
	
	We next calculate the complexity of the two conditions on the right-hand side of \eqref{eq:s3-rank-iso-characterization}.
	
	The map $\varphi_{N,\bar w}$ is onto exactly when every generator $a_nN$ of $G_N$ belongs to the subgroup generated by $w_0N,\ldots,w_{k-1}N$. Equivalently,
	\begin{equation}\label{eq:s3-rank-30}
		\forall n\in\omega\ \exists u\in F_k\ 
		\bigl[a_n^{-1}u(\bar w)\in N\bigr].
	\end{equation}
	Indeed, if $\varphi_{N,\bar w}$ is onto, then for every $n$ there is some $u\in F_k$ with $\varphi_{N,\bar w}(u)=a_nN$. Since $\varphi_{N,\bar w}(u)=u(\bar w)N$, this is equivalent to $a_n^{-1}u(\bar w)\in N$. In the other direction, \eqref{eq:s3-rank-30} says that every $a_nN$ belongs to the image of $\varphi_{N,\bar w}$. Since the elements $a_nN$, $n\in\omega$, generate $G_N$, the map is surjective. Formula \eqref{eq:s3-rank-30} is a countable intersection of open conditions, so $\{N:\varphi_{N,\bar w}\text{ is onto}\}\in\bfpi^0_2$.
	
	The equality $\ker(\varphi_{N,\bar w})=K_H$ is equivalent to
	\begin{equation}\label{eq:s3-rank-31}
		\forall u\in F_k\ 
		\bigl[
		u(\bar w)\in N
		\longleftrightarrow
		u\in K_H
		\bigr].
	\end{equation}
	For fixed $u\in F_k$, the statement $u\in K_H$ has a fixed truth value, while $u(\bar w)\in N$ is clopen in $\mathrm{ctblgrp}$. The condition inside the brackets in \eqref{eq:s3-rank-31} is accordingly clopen for each $u$, and the collection of $N$ satisfying \eqref{eq:s3-rank-31} is closed. Thus by \eqref{eq:s3-rank-iso-characterization}, we have $[H]_{\cong}\in\bfsgm^0_3$. This proves (1).
	
	We now prove (2). A quotient $G_N$ is finitely generated exactly when there are $k<\omega$ and $\bar w\in F_\omega^k$ such that $\varphi_{N,\bar w}:F_k\to G_N$ is onto. For each fixed $k$ and $\bar w$, the corresponding surjectivity condition belongs to $\bfpi^0_2$ by \eqref{eq:s3-rank-30}, and then $\mathsf{FG}\in\bfsgm^0_3$.
	
	There are only countably many finite group presentations, so there are only countably many isomorphism types of finitely presented groups. Part~(1) gives $\mathsf{FP}\in\bfsgm^0_3$, since $\mathsf{FP}$ is the union of the classes $[H]_{\cong}$ over all finitely presented groups $H$. Since \STRP is invariant under group isomorphism, the same argument gives $\mathsf{FP}\cap\mathsf{STRP}\in\bfsgm^0_3$ by restricting the union to the finitely presented groups having \STRP.
	
	We finish with $\bfsgm^0_3$-hardness. Let $\Theta:2^{\omega\times\omega}\to\mathrm{ctblgrp}$ be the continuous map constructed above. Lemma~\ref{lem:carson-continuous} and Lemma~\ref{lem:group-to-nsub} give $G_{\Theta(x)}\cong F_{m(x)+2}$ when $x\in\mathsf S_3$, and $G_{\Theta(x)}\cong F_\omega$ when $x\notin\mathsf S_3$.
	
	If $x\in\mathsf S_3$, then $G_{\Theta(x)}$ is a finite-rank free group. It is finitely generated and finitely presented, and it has \STRP. If $x\notin\mathsf S_3$, then $G_{\Theta(x)}\cong F_\omega$. The group $F_\omega$ is not finitely generated, so it is not finitely presented, and it does not have \STRP.
	
	We have, for each $\mathcal C\in\{\mathsf{FG},\mathsf{FP},\mathsf{FP}\cap\mathsf{STRP}\}$,
	$x\in\mathsf S_3\longleftrightarrow\Theta(x)\in\mathcal C$. The set $\mathsf S_3$ is $\bfsgm^0_3$-complete and $\Theta$ is continuous, so each of the three classes is $\bfsgm^0_3$-hard. Together with the upper bounds above, $\mathsf{FG}$, $\mathsf{FP}$, and $\mathsf{FP}\cap\mathsf{STRP}$ are $\bfsgm^0_3$-complete.
\end{proof}

Define
\[
\mathsf{VF}
=
\{N\in\mathrm{ctblgrp}:
G_N\text{ is finitely generated and virtually free}\}.
\]

\begin{cor}\label{cor:vf-complete}
	The collection $\mathsf{VF}$ is $\bfsgm^0_3$-complete and
	$\mathsf{VF}\subseteq\mathsf{STRP}$.
\end{cor}

\begin{proof}
	Every finitely generated virtually free group is finitely presented. Let $G$ be a finitely generated virtually free group, and choose a free subgroup $F\leq G$ of finite index. As in the proof of \ref{thm:locally-virtually-free}, let
	$K=\bigcap_{g\in G}gFg^{-1}$ be the normal core of $F$ in $G$. Since $F$ has finite index in $G$, it has only finitely many conjugates, so $K\triangleleft G$ also has finite index. Moreover, $K\leq F$. The subgroup $F$ is finitely generated because it has finite index in the finitely generated group $G$. Since $K$ has finite index in $F$, it is also finitely generated. By the Nielsen--Schreier theorem, $K$ is free \cite[Chapter~I, Section~3]{lyndonschupp1977}. Thus $K$ is a free group of finite rank.

	Write $Q=G/K$. Since $Q$ is finite, choose representatives
	$t_q\in G$ for $q\in Q$, with $t_{1_Q}=1_G$. Let $X$ be a finite
	free basis of $K$. For every $q\in Q$ and $x\in X$, the element
	$t_qxt_q^{-1}$ belongs to $K$ because $K\triangleleft G$; let
	$w_{q,x}(X)$ be the unique reduced word over $X\cup X^{-1}$
	representing it. For every $q,r\in Q$, the elements $t_qt_r$ and
	$t_{qr}$ represent the same coset of $K$, so
	$t_qt_rt_{qr}^{-1}\in K$; let $c_{q,r}(X)$ be the unique reduced
	word over $X\cup X^{-1}$ representing this element.
	
	Consider the presentation
	\begin{equation}\label{eq:s3-rank-34}
		\begin{aligned}
			\left\langle\right.
			X,(T_q)_{q\in Q}
			\ |\
			T_{1_Q}=1,\
			(T_qxT_q^{-1}&=w_{q,x}(X))_{\ q\in Q,\ x\in X},\\
			(T_qT_r&=c_{q,r}(X)T_{qr})_{q,r\in Q}
			\left.\right\rangle .
		\end{aligned}
	\end{equation}
	Let $P$ denote the group defined by the presentation in
	\eqref{eq:s3-rank-34}. Since all the defining relations hold in $G$,
	the assignments $x\mapsto x$ for $x\in X$ and $T_q\mapsto t_q$ for
	$q\in Q$ induce a homomorphism $\pi:P\to G$. This map is surjective.
	Indeed, for every $g\in G$, if $q=gK\in Q$, then
	$gt_q^{-1}\in K$, so $g=kt_q$ for some $k\in K$. Since $X$ is a
	basis of $K$, the element $k$ is represented by a word over
	$X\cup X^{-1}$ and belongs to the image of $\pi$.
	
	We prove that $\pi$ is injective. We first show that every element of $P$ has the form $u(X)T_q$, where $u(X)$ is a word over $X\cup X^{-1}$ and $q\in Q$. The relations imply $T_qx^\varepsilon=w_{q,x}(X)^\varepsilon T_q$ for $q\in Q$, $x\in X$, and $\varepsilon\in\{1,-1\}$, while $T_{q^{-1}}T_q=c_{q^{-1},q}(X)$ gives $T_q^{-1}=c_{q^{-1},q}(X)^{-1}T_{q^{-1}}$. Starting with an arbitrary word in the generators of $P$, first eliminate every occurrence of $T_q^{-1}$ by the latter relation. Then use $T_qx^\varepsilon=w_{q,x}(X)^\varepsilon T_q$ to move every letter from $X\cup X^{-1}$ to the left of every $T_q$. The remaining product of the $T_q$'s is reduced successively by $T_qT_r=c_{q,r}(X)T_{qr}$. After finitely many applications of these relations, the word has the form $u(X)T_q$.

	Let $p\in\ker(\pi)$, and choose such a representation
	$p=u(X)T_q$. Its image in $G$ is $u(X)t_q=1_G$. As $ u(X)\in K $, we have $ q=1_Q $, and then $p=u(X)$ in $P$. The equality
	$\pi(p)=1_G$ now says that the word $u(X)$ represents the identity in
	$K$. Since $X$ is a free basis of $K$, $u(X)$ freely reduces to the
	empty word. The same free reduction is valid in $P$, so $p=1_P$.
	Thus $\ker(\pi)=\{1_P\}$, and $\pi:P\to G$ is an isomorphism.
	
	Both $X$ and $Q$ are finite. The presentation
	\eqref{eq:s3-rank-34} has finitely many generators and finitely many
	relations, so $G$ is finitely presented.

	Thus there are only countably many finitely generated virtually free
	isomorphism types. Proposition~\ref{prop:sigma03-loci}(1) gives
	$\mathsf{VF}\in\bfsgm^0_3$.
	
	The reduction $\Theta$ proves hardness, since its positive outputs are
	finite-rank free groups and its negative output is $F_\omega$.
	Finally,
	$\mathsf{VF}\subseteq\mathsf{STRP}$ follows from
	Corollary~\ref{cor:virtually-free}.
\end{proof}

Define
\[
\mathsf{LVF}
=
\{N\in\mathrm{ctblgrp}:G_N\text{ is locally virtually free}\}.
\]
Theorem~\ref{thm:locally-virtually-free} and
Corollary~\ref{cor:virtually-free} give $ \mathsf{STRP}\cap\mathsf{LVF}=\mathsf{VF} $.

We can now state the consequences for the unresolved Borel rank separately.

\begin{cor}\label{cor:rank-dichotomy}
	The collection $\mathsf{STRP}$ does not belong to $\bfpi^0_3$.
	Its Borel rank is $3$ or $4$. If
	$\mathsf{STRP}\in\bfsgm^0_3$, then it is
	$\bfsgm^0_3$-complete.
\end{cor}

\begin{proof}
	If $\mathsf{STRP}$ were $\bfpi^0_3$, then its continuous preimage
	under $\Theta$ would place the $\bfsgm^0_3$-complete set
	$\mathsf S_3$ in $\bfpi^0_3$, contrary to the strictness of the
	Borel hierarchy.
	Theorem~\ref{thm:pi04} gives the rank alternative.
	If $\mathsf{STRP}\in\bfsgm^0_3$, then
	Proposition~\ref{prop:sigma03-hard} makes it
	$\bfsgm^0_3$-complete.
\end{proof}

\begin{cor}\label{cor:rank-four-consequences}
	If $\mathsf{STRP}\notin\bfsgm^0_3$, then:
	\begin{enumerate}[label=(\arabic*),font=\upshape]
		\item some group with \STRP is not locally virtually free;
		\item $\mathsf{STRP}$ is not contained in any countable union of
		isomorphism classes of finitely generated groups.
	\end{enumerate}
	Consequently, either some group with \STRP is not finitely generated,
	or there are uncountably many finitely generated groups with \STRP
	up to isomorphism. In the latter case, some finitely generated group
	with \STRP admits no recursive presentation. In particular, some
	non-finitely-presented group has \STRP.
\end{cor}

\begin{proof}
	For (1), if every group with \STRP were locally virtually free, then we would have $ \mathsf{STRP}=\mathsf{VF} $, contradicting $\mathsf{VF}\in\bfsgm^0_3$.
	
	For (2), suppose that
	$\mathsf{STRP}\subseteq\bigcup_{n\in\omega}[H_n]_{\cong}$ for some
	finitely generated groups $H_n$. Since \STRP is invariant under group
	isomorphism,
	\[
	\mathsf{STRP}
	=
	\bigcup_{\substack{n\in\omega\\ H_n\text{ has }\STRP}}
	[H_n]_{\cong}.
	\]
	Each $[H_n]_{\cong}$ belongs to $\bfsgm^0_3$ by
	Proposition~\ref{prop:sigma03-loci}(1), so
	$\mathsf{STRP}\in\bfsgm^0_3$, contrary to the hypothesis.
	
	If every group with \STRP is finitely generated, (2) therefore gives
	uncountably many such isomorphism types. There are only countably many
	isomorphism types of finitely generated recursively presentable groups,
	so one of them admits no recursive presentation. Finally, a finitely
	presented group is finitely generated and recursively presentable.
\end{proof}

	\section{Recursive consequences}
	Carrasco-Vargas proved that a recursively presented finitely generated group
	admitting an SFT of nonzero Medvedev degree does not have \STRP
	\cite[Theorem~1.1]{carrasco2026}. A key ingredient is that every
	projectively isolated subshift over such a group has decidable language
	\cite[Proposition~3.1]{carrasco2026}. We extend this effective argument to
	quotients of $F_\omega$ by recursively enumerable normal subgroups, allowing
	countably many named generators. The effective compactness argument below is
	the countable-generator analogue of the finite-obstruction argument in
	\cite[Lemma~3.2]{carrasco2026}, while
	Proposition~\ref{prop:pi-decidable-countable} extends
	\cite[Proposition~3.1]{carrasco2026}.
	
	We use the standard terminology of recursively presented Polish spaces. Fix a recursive enumeration $q_n$ of the positive rational numbers. A \textbf{recursive presentation} of a Polish metric space $(X,d)$ is a dense sequence $(x_n)_{n\in\omega}$ such that the relations $d(x_i,x_j)<q_k$ and $d(x_i,x_j)\leq q_k$ on $\omega^3$ are recursive. A Polish metric space equipped with such a presentation is called \textbf{recursively presented}. The presentation determines an effective enumeration $(U_n^X)_{n\in\omega}$ of a basis of open balls. A subset of $X$ is \textbf{semirecursive} if it is an effectively enumerable union of members of this basis. For the discrete spaces used below, equipped with their fixed
	discrete-metric presentations, semirecursive subsets are exactly
	recursively enumerable subsets. If $X$ and $Y$ are recursively presented Polish spaces, a function $f:X\to Y$ is \textbf{recursive} if the relation $\{(n,x)\in\omega\times X:f(x)\in U_n^Y\}$ is semirecursive in $\omega\times X$.
	
	Give $F_\omega=\langle a_0,a_1,\ldots\rangle$ the discrete metric, and fix a recursive enumeration without repetitions of its finite reduced words as a recursive presentation of $F_\omega$. With this presentation, multiplication $F_\omega\times F_\omega\to F_\omega$, inversion $F_\omega\to F_\omega$, and substitution of finitely many elements of $F_\omega$ into a word are recursive. We use the term \textbf{recursively enumerable} for semirecursive subsets of $F_\omega$. Every finite alphabet is regarded as a recursively presented finite discrete space, and finite powers of these spaces and of $F_\omega$ carry their product recursive presentations.
	
	\begin{defn}\label{def:recursive-countable-presentation}
		A \textbf{recursive countable presentation} is a quotient $G_N=F_\omega/N$, where $N\triangleleft F_\omega$ is recursively enumerable. Let $\pi_N:F_\omega\to G_N$ be the quotient map. For $x\in A^{G_N}$, define $\widehat x=x\circ\pi_N:F_\omega\to A$, so $\widehat x(w)=x(wN)$. The configuration $x$ is \textbf{$N$-recursive} if $\widehat x$ is recursive.
	\end{defn}
	
	The definition does not require the word problem of $G_N$ to be recursive. The relation $\{(u,v)\in F_\omega^2:uN=vN\}$ is recursively enumerable, since $uN=vN$ if and only if $u^{-1}v\in N$, and the map $(u,v)\mapsto u^{-1}v$ is recursive.
	
	A \textbf{mass problem} is a subset of Baire space $\omega^\omega$; we use the same terminology for subsets of Cantor space. For mass problems $P,Q$, we say that $P$ is \textbf{Medvedev reducible} to $Q$, and write $P\leq_M Q$, if there is a partial recursive function $\Phi:\omega^\omega\rightharpoonup\omega^\omega$ such that $Q\subseteq\operatorname{dom}(\Phi)$ and $\Phi(y)\in P$ for every $y\in Q$. We write $P\equiv_M Q$ if $P\leq_M Q$ and $Q\leq_M P$; the equivalence classes under $\equiv_M$ are the \textbf{Medvedev degrees}. The least Medvedev degree consists exactly of the mass problems containing a recursive element. In particular, a nonempty subshift has Medvedev degree zero exactly when it contains a recursive configuration in the pullback; see \cite[Section~3]{barbiericarrasco2024} and \cite[Section~2.4]{carrasco2026}.
	
	For a finite symbolic code $c=(A,\bar f,\mathcal F)$, where $\bar f=(f_0,\ldots,f_{m-1})\in F_\omega^m$, a finite tuple $\bar k=(k_0,\ldots,k_{\ell-1})\in F_\omega^\ell$, and $q\in A^\ell$, write
	\begin{equation}\label{eq:s4-01}
		Y_c(N;\bar k,q)=\{x\in Y_c(N):(x(k_i))_{i<\ell}=q\}.
	\end{equation}
	
	Let $\mathscr P$ consist of all triples $(c,\bar k,q)$ such that
	$c=(A,\bar f,\mathcal F)$ is a finite symbolic code,
	$\bar k\in F_\omega^\ell$ for some $\ell\in\omega$, and
	$q\in A^\ell$. We equip $\mathscr P$ with the discrete metric and fix
	a recursive presentation induced by the fixed recursive presentation
	of $F_\omega$. With this presentation, the maps extracting
	$A,\bar f,\mathcal F,\bar k$ and $q$ from $(c,\bar k,q)$ are
	recursive, as are the length and coordinate maps on finite tuples and
	membership in finite families of patterns.
	
	\subsection{Effective consequences of projective isolation}
	
	The next lemma adapts the effective compactness argument of
	\cite[Lemma~3.2]{carrasco2026} to recursive countable presentations.
	
	\begin{lem}\label{lem:effective-compactness}
		Let $N\triangleleft F_\omega$ be recursively enumerable. The set
		$\{(c,\bar k,q)\in\mathscr P:Y_c(N;\bar k,q)=\emptyset\}$
		is recursively enumerable.
	\end{lem}
	
	\begin{proof}
		Fix $(c,\bar k,q)\in\mathscr P$, and write
		$c=(A,\bar f,\mathcal F)$,
		$\bar f=(f_0,\ldots,f_{m-1})$, and
		$\bar k=(k_0,\ldots,k_{\ell-1})$.
		For $u,t,w\in F_\omega$, define the clopen subsets of
		$A^{F_\omega}$ by
		$C_{u,t}=\{x:x(u)=x(ut)\}$,
		$L_{c,w}=\{x:(x(w^{-1}f_i))_{i<m}\in\mathcal F\}$, and
		$P_{\bar k,q}=\{x:(x(k_i))_{i<\ell}=q\}$.
		Then
		\begin{equation}\label{eq:s4-effective-intersection}
			Y_c(N;\bar k,q)
			=
			P_{\bar k,q}
			\cap
			\bigcap_{t\in N}\bigcap_{u\in F_\omega}C_{u,t}
			\cap
			\bigcap_{w\in F_\omega}L_{c,w}.
		\end{equation}
		
		For $r,s\in\omega$, $\bar u=(u_0,\ldots,u_{r-1})\in F_\omega^r$,
		$\bar t=(t_0,\ldots,t_{r-1})\in F_\omega^r$, and
		$\bar w=(w_0,\ldots,w_{s-1})\in F_\omega^s$, let
		$B(c,\bar k,q;\bar u,\bar t,\bar w)$ be the finite intersection
		\[
		P_{\bar k,q}\cap\bigcap_{j<r}C_{u_j,t_j}
		\cap\bigcap_{j<s}L_{c,w_j}.
		\]
		Each set occurring in \eqref{eq:s4-effective-intersection} is closed in the compact space $A^{F_\omega}$. Suppose that $Y_c(N;\bar k,q)=\emptyset$. Then the complements of $P_{\bar k,q}$, $C_{u,t}$ for $u\in F_\omega$ and $t\in N$, and $L_{c,w}$ for $w\in F_\omega$ form an open cover of $A^{F_\omega}$. By compactness, a finite subfamily already covers $A^{F_\omega}$. The converse is immediate from \eqref{eq:s4-effective-intersection}. Thus
		\begin{equation}\label{eq:s4-effective-certificate}
			\begin{aligned}
				Y_c(N;\bar k,q)=\emptyset
				\longleftrightarrow\
				&\exists r,s\in\omega\
				\exists\bar u\in F_\omega^r\
				\exists\bar t\in N^r\
				\exists\bar w\in F_\omega^s\\
				&\bigl[
				B(c,\bar k,q;\bar u,\bar t,\bar w)=\emptyset
				\bigr].
			\end{aligned}
		\end{equation}
		
		Let $\operatorname{Fn}(F_\omega,\omega)$ denote the recursively
		presented discrete space of finite partial functions from $F_\omega$
		to $\omega$. For $c=(A,\bar f,\mathcal F), r,s\in\omega$, $\bar u$,
		$\bar t$, and $\bar w$, define
		\[
		D=
		\{k_i:i<\ell\}
		\cup
		\{u_j,u_jt_j:j<r\}
		\cup
		\{w_j^{-1}f_i:j<s,\ i<m\}.
		\]
		By the compactness argument above, the set $ \{(c,\bar k,q)\in\mathscr P:Y_c(N;\bar k,q)=\emptyset\} $ equals
		\begin{equation}\label{eq:s4-effective-re}
			\begin{aligned}
				\bigcup_{r,s\in\omega}
				\bigcup_{\bar u\in F_\omega^r}
				\bigcup_{\bar t\in F_\omega^r}
				\bigcup_{\bar w\in F_\omega^s}
				\Biggl(
				\bigcap_{j<r}
				\Bigl\{
				&(c,\bar k,q)\in\mathscr P:t_j\in N
				\Bigr\}\\
				\cap
				\Biggl[
				\mathscr P\setminus
				\bigcup_{p\in\operatorname{Fn}(F_\omega,\omega)}
				\Biggl\{
				&(c,\bar k,q)\in\mathscr P:
				\operatorname{dom}(p)=D \
				\wedge \ 
				\operatorname{ran}(p)\subseteq A\\
				&\wedge \ 
				p\upharpoonright\bar k=q \ 
				\wedge \ 
				p\upharpoonright\bar u
				=
				p\upharpoonright(\bar u\bar t)\\
				&\wedge \ 
				\forall j<s\
				\Bigl[
				p\upharpoonright(w_j^{-1}\bar f)\in\mathcal F
				\Bigr]
				\Biggr\}
				\Biggr]
				\Biggr).
			\end{aligned}
		\end{equation}
		
		For fixed $r,s,\bar u,\bar t,\bar w$, the set in square brackets in
		\eqref{eq:s4-effective-re} is recursive, uniformly in these parameters.
		Indeed, the finite set $D$ is obtained recursively from
		$(c,\bar k,q,\bar u,\bar t,\bar w)$. For
		$p\in\operatorname{Fn}(F_\omega,\omega)$, each condition
		$\operatorname{dom}(p)=D$, $\operatorname{ran}(p)\subseteq A$,
		$p\upharpoonright\bar k=q$,
		$p\upharpoonright\bar u=p\upharpoonright(\bar u\bar t)$, and
		$\forall j<s\,[p\upharpoonright(w_j^{-1}\bar f)\in\mathcal F]$ is
		recursive in the displayed parameters. Moreover, the first two
		conditions restrict $p$ to the finite set $A^D$, which is obtained
		recursively from $A$ and $D$. The inner union over
		$p\in\operatorname{Fn}(F_\omega,\omega)$ is accordingly recursive, and
		so is its complement in $\mathscr P$.
		
		For each $r<\omega$ and $j<r$, the relation
		$\{(\bar t,c,\bar k,q)\in F_\omega^r\times\mathscr P:t_j\in N\}$ is
		recursively enumerable, so is
		its finite intersection. Finally, the outer unions in
		\eqref{eq:s4-effective-re} preserve recursive enumerability. So
		$\{(c,\bar k,q)\in\mathscr P:Y_c(N;\bar k,q)=\emptyset\}$ is
		recursively enumerable.
	\end{proof}

The next proposition is the recursive-countable analogue of
\cite[Proposition~3.1]{carrasco2026}. Its final assertion, producing an
$N$-recursive configuration from the decidable finite tuple-language, is
analogous to the computable-point consequence of
\cite[Proposition~3.3]{carrasco2026}.

\begin{prop}\label{prop:pi-decidable-countable}
	Let $G_N$ have a recursive countable presentation, and let
	$X\subseteq A^{G_N}$ be projectively isolated. Then
	\[
	\mathcal L_X = \left\{
	(\bar k,q):
	\exists \ell\in\omega, \bar k\in F_\omega^\ell,\
	q\in A^\ell,\
	\exists x\in X\
	[(x(k_iN))_{i<\ell}=q]
	\right\}
	\]
	is a recursive subset of
	$\bigcup_{\ell<\omega}F_\omega^\ell\times A^\ell$ with its natural recursive discrete presentation. Equivalently, the finite tuple-language of $X$ is decidable. Moreover, $X$ contains an $N$-recursive configuration.
\end{prop}

\begin{proof}
	Choose a witness from Proposition~\ref{prop:projectively-isolated}. Thus there are a finite alphabet $B$, a finite set $F\subseteq G_N$, a globally realizable triple $(B,F,\mathcal F)$, and a surjection $p:B\to A$ such that, writing $W=[F,\mathcal F]_{G_N}$, one has $W_F=\mathcal F$, $p^{G_N}(W)=X$, and
	$U_F=\mathcal F\longrightarrow p^{G_N}(U)=X$
	for every $U\in\Sub_{G_N}(B)$.
	
	Enumerate $F=\{g_0,\ldots,g_{m-1}\}$ and choose $f_i\in F_\omega$ with $f_iN=g_i$ for $i<m$. Write $\bar f=(f_0,\ldots,f_{m-1})$ and identify $B^F$ with $B^m$ by $r\mapsto(r(g_i))_{i<m}$; under this identification, retain the notation $\mathcal F\subseteq B^m$. Let $d=(B,\bar f,\mathcal F)$ and $V=Y_d(N)$. Then $V=\pi_N^*(W)$. Write $p$ also for the coordinatewise map $B^{F_\omega}\to A^{F_\omega}$, and let $\widehat X=\pi_N^*(X)$. Since $p\circ\pi_N^*=\pi_N^*\circ p^{G_N}$, the witness gives
	\begin{equation}\label{eq:s4-02}
		V_{\bar f}=\mathcal F,\
		p(V)=\widehat X,\
		Z_{\bar f}=\mathcal F\longrightarrow p(Z)=\widehat X
	\end{equation}
	for every subshift $Z\subseteq V$.

	For the finite tuple-language $\mathcal L_X$ in the statement, fix $\ell<\omega$, $\bar k\in F_\omega^\ell$, and $q\in A^\ell$. Let
	$\mathcal B_q=\{b\in B^\ell:(p(b_i))_{i<\ell}=q\}$ and
	$V_q=\{y\in V:(p(y(k_i)))_{i<\ell}=q\}$.
	Then $V_q=\bigcup_{b\in\mathcal B_q}Y_d(N;\bar k,b)$, and
	\begin{equation}\label{eq:s4-03}
		(\bar k,q)\notin\mathcal L_X \
		\longleftrightarrow \
		V_q=\emptyset \
		\longleftrightarrow \
		\forall b\in\mathcal B_q\
		[Y_d(N;\bar k,b)=\emptyset].
	\end{equation}
	The finite set $\mathcal B_q$ is obtained recursively from $q$. By Lemma~\ref{lem:effective-compactness}, the relation $Y_d(N;\bar k,b)=\emptyset$ is recursively enumerable in $(\bar k,b)$. Finite intersections of recursively enumerable relations are recursively enumerable, so \eqref{eq:s4-03} shows that the complement of $\mathcal L_X$ is recursively enumerable.
	
	Let $d^{\neg(\bar k,q)}$ be the code defined in \eqref{eq:s3-coding-04}, and write
	$V^{\neg(\bar k,q)}=Y_{d^{\neg(\bar k,q)}}(N)$. Finite tuple concatenation and membership in the finite family $\mathcal F^{\neg q}$ are recursive, so the map $(\bar k,q)\mapsto d^{\neg(\bar k,q)}$ is recursive. We claim that
	\begin{equation}\label{eq:s4-04}
		(\bar k,q)\in\mathcal L_X
		\longleftrightarrow
		(V^{\neg(\bar k,q)})_{\bar f}\neq\mathcal F.
	\end{equation}
	Suppose first that $(\bar k,q)\notin\mathcal L_X$. Then
	$\forall x\in X\ [(x(k_iN))_{i<\ell}\neq q]$, and hence
	$\forall \widehat x\in\widehat X\ [(\widehat x(k_i))_{i<\ell}\neq q]$.
	Since $\widehat X=p(V)$ is $F_\omega$-invariant, for every $y\in V$ and
	$w\in F_\omega$ one has $\sigma_{F_\omega}(w)p(y)\in\widehat X$, so
	\[
	(p(y(w^{-1}k_i)))_{i<\ell}
	=
	((\sigma_{F_\omega}(w)p(y))(k_i))_{i<\ell}
	\neq q.
	\]
	Thus every $y\in V$ satisfies the additional rules defining
	$V^{\neg(\bar k,q)}$. Therefore $V\subseteq V^{\neg(\bar k,q)}$; the reverse
	inclusion follows from the definition of $d^{\neg(\bar k,q)}$. Hence
	$V^{\neg(\bar k,q)}=V$, and \eqref{eq:s4-02} gives
	$(V^{\neg(\bar k,q)})_{\bar f}=V_{\bar f}=\mathcal F$.
	
	Conversely, suppose $(\bar k,q)\in\mathcal L_X$ and $(V^{\neg(\bar k,q)})_{\bar f}=\mathcal F$. Since $\mathcal F=V_{\bar f}\neq\emptyset$, the subshift $V^{\neg(\bar k,q)}$ is nonempty. Moreover, $V^{\neg(\bar k,q)}\subseteq V$, so \eqref{eq:s4-02} gives $p(V^{\neg(\bar k,q)})=\widehat X$. By the defining rules of $d^{\neg(\bar k,q)}$, no point of $p(V^{\neg(\bar k,q)})$ realizes $q$ on $\bar k$, whereas $(\bar k,q)\in\mathcal L_X$ says that $\widehat X$ contains such a point. This contradiction proves \eqref{eq:s4-04}.
	
	Since $V^{\neg(\bar k,q)}\subseteq V$, one has $(V^{\neg(\bar k,q)})_{\bar f}\subseteq\mathcal F$. Hence \eqref{eq:s4-04} is equivalent to
	\begin{equation}\label{eq:s4-05}
		(\bar k,q)\in\mathcal L_X
		\longleftrightarrow
		\exists a\in\mathcal F\
		[Y_{d^{\neg(\bar k,q)}}(N;\bar f,a)=\emptyset].
	\end{equation}
	For each $a\in\mathcal F$, Lemma~\ref{lem:effective-compactness} and the recursiveness of $(\bar k,q)\mapsto d^{\neg(\bar k,q)}$ show that the relation on the right of \eqref{eq:s4-05} is recursively enumerable. Since $\mathcal F$ is finite, $\mathcal L_X$ is recursively enumerable. By \eqref{eq:s4-03}, its complement is recursively enumerable as well. Thus $\mathcal L_X$ is recursive.
	
Let $(w_n)_{n<\omega}$ be the recursive enumeration of $F_\omega$ underlying its fixed recursive presentation, and fix the natural ordering of the finite alphabet $A$. Recursively define
\[
a_n=\min\{a\in A:
((w_0,\ldots,w_n),(a_0,\ldots,a_{n-1},a))\in\mathcal L_X\}.
\]
This is well defined. Indeed, if $(a_0,\ldots,a_{n-1})$ is realized by some $x\in X$ on $(w_0N,\ldots,w_{n-1}N)$, then $a=x(w_nN)$ gives an admissible extension. Since $\mathcal L_X$ is recursive and $A$ is finite, the sequence $(a_n)_{n<\omega}$ is recursive.

For $n<\omega$, let
$C_n=\{x\in X:\forall i\leq n\ [x(w_iN)=a_i]\}$.
Each $C_n$ is nonempty and closed, and $C_{n+1}\subseteq C_n$. Compactness of $X$ gives
$\bigcap_{n<\omega}C_n\neq\emptyset$.
Choose $x$ in this intersection. Then $\widehat x(w_n)=a_n$ for every $n<\omega$. Since $(w_n)$ is the recursive presentation of the discrete space $F_\omega$ and $(a_n)$ is recursive, the map $\widehat x:F_\omega\to A$ is recursive. Thus $x$ is $N$-recursive.
\end{proof}

	\begin{thm}\label{thm:recursive-countable-obstruction}
		Let $G_N$ have a recursive countable presentation and \STRP. Then every nonempty $G_N$-SFT contains an $N$-recursive configuration.
	\end{thm}
	
	\begin{proof}
		Let $X\subseteq A^{G_N}$ be a nonempty SFT. If $|A|=1$, then $X=A^{G_N}$ and its unique configuration is $N$-recursive. Assume $|A|\geq2$.
		
		By \cite[Lemma~2.14]{doucha2024}, there is a finite $F\subseteq G_N$ such that every subshift in $\mathcal N_X^F$ is contained in $X$. Since $G_N$ has \STRP, projectively isolated subshifts are dense in $\Sub_{G_N}(A)$ by \cite[Theorem~3.1]{doucha2024}. Choose a projectively isolated $Y\in\mathcal N_X^F$. Then $Y\subseteq X$, and Proposition~\ref{prop:pi-decidable-countable} gives an $N$-recursive configuration in $Y$, hence in $X$.
	\end{proof}

	\subsection{Subgroup obstructions}
	
	We now turn the preceding effective obstruction into obstructions coming
	from finitely generated subgroups. Let
	$H=\langle u_0N,\ldots,$ $u_{m-1}N\rangle$   $\leq G_N$, where
	$\bar u=(u_0,\ldots,u_{m-1})\in F_\omega^m$. Write
	$F_m=\langle s_0,\ldots,s_{m-1}\rangle$, and let
	$\rho_{\bar u}:F_m\to H$ be the epimorphism defined by
	$\rho_{\bar u}(s_i)=u_iN$. If
	$\tau_{\bar u}:F_m\to F_\omega$ is defined by
	$\tau_{\bar u}(s_i)=u_i$, then
	$K_{N,\bar u}:=\ker(\rho_{\bar u})=\tau_{\bar u}^{-1}(N)$.
	Since $\tau_{\bar u}$ is recursive, $K_{N,\bar u}$ is recursively enumerable whenever $N$ is recursively enumerable, uniformly in $\bar u$. A configuration $x\in A^H$ is called \emph{$K_{N,\bar u}$-recursive} if the pullback
	$x\circ\rho_{\bar u}:F_m\to A$ is recursive.
	
	\begin{cor}\label{cor:recursive-subgroup-obstruction}
		Let $G_N$ have a recursive countable presentation, and let
		$H=\langle u_0N,\ldots,u_{m-1}N\rangle\leq G_N$. If $H$ admits a nonempty SFT containing no $K_{N,\bar u}$-recursive configuration, then $G_N$ does not have \STRP.
	\end{cor}
	
	\begin{proof}
		Let $X=[F,\mathcal F]_H\subseteq A^H$ be such an SFT, and let
		$X^\uparrow=[F,\mathcal F]_{G_N}$ be its free extension. By
		Lemma~\ref{lem:free-extension}, $X^\uparrow$ is nonempty and
		$y\uphar H\in X$ for every $y\in X^\uparrow$.
		
		Suppose that $y\in X^\uparrow$ is $N$-recursive, and let
		$x=y\uphar H$. For every $v\in F_m$,
		$(x\circ\rho_{\bar u})(v)
		=y(\tau_{\bar u}(v)N)
		=\widehat y(\tau_{\bar u}(v))$.
		Since $\widehat y$ and $\tau_{\bar u}$ are recursive,
		$x\circ\rho_{\bar u}$ is recursive. Thus $x$ is
		$K_{N,\bar u}$-recursive, contrary to the choice of $X$.
		Therefore $X^\uparrow$ contains no $N$-recursive configuration.
		Theorem~\ref{thm:recursive-countable-obstruction} now implies that
		$G_N$ does not have \STRP.
	\end{proof}
	
	\begin{cor}\label{cor:recursive-product-obstruction}
		Let $G_N$ have a recursive countable presentation. If $G_N$ contains a subgroup isomorphic to $H_1\times H_2$, where $H_1$ and $H_2$ are infinite finitely generated groups, then $G_N$ does not have \STRP.
	\end{cor}
	
	\begin{proof}
		Let $J\leq G_N$ be a subgroup isomorphic to $H_1\times H_2$, and identify
		$J$ with this direct product. Choose a finite generating tuple
		$(j_0,\ldots,j_{m-1})$ for $J$, and choose
		$u_0,\ldots,u_{m-1}\in F_\omega$ such that $u_iN=j_i$ for every $i<m$.
		Write $\bar u=(u_0,\ldots,u_{m-1})$, and let
		$\rho_{\bar u}:F_m\to J$ be the epimorphism defined by
		$\rho_{\bar u}(s_i)=j_i=u_iN$. Then
		$K_{N,\bar u}=\ker(\rho_{\bar u})$.
		
		By \cite[Theorem~5.8]{barbiericarrasco2024}, there is a nonempty
		$J$-SFT $X\subseteq A^J$ with nonzero Medvedev degree. By definition,
		the Medvedev degree of $X$ is the degree of its pullback
		\[
		\widehat X
		=
		\{x\circ\rho_{\bar u}:x\in X\}
		\subseteq A^{F_m}.
		\]
		Since a mass problem has degree zero exactly when it contains a recursive
		element, the nonzero degree of $X$ implies
		$\forall x\in X\ [x\circ\rho_{\bar u}\text{ is not recursive}]$.
		Thus $X$ contains no $K_{N,\bar u}$-recursive configuration.
		Corollary~\ref{cor:recursive-subgroup-obstruction} applies, so $G_N$
		does not have \STRP.
	\end{proof}
	
	\begin{rem}\label{rem:recursive-literature-status}
		For comparison with the finitely generated setting, Carrasco-Vargas works
		with finitely generated groups throughout \cite[Section~2.1]{carrasco2026}. In that setting,
		\cite[Proposition~3.1]{carrasco2026} is the finitely generated analogue of
		Proposition~\ref{prop:pi-decidable-countable}, while
		\cite[Theorem~1.1]{carrasco2026} gives the obstruction from a nonempty SFT of nonzero Medvedev degree. For a finitely generated ambient group, Corollary~\ref{cor:recursive-subgroup-obstruction} also follows from
		\cite[Corollary~4.4]{barbiericarrasco2024} together with that theorem.
		The direct-product input in Corollary~\ref{cor:recursive-product-obstruction} is
		\cite[Theorem~5.8]{barbiericarrasco2024}. The arguments above allow the ambient group to have countably many named generators.
	\end{rem}
	
	\begin{cor}\label{cor:slnq}
		For every $n\geq2$, the group $\mathrm{SL}_n(\bbq)$ does not have \STRP.
	\end{cor}
	
	\begin{proof}
		We first give $\mathrm{SL}_n(\bbq)$ a recursive countable presentation, which is routine.
		Exact arithmetic and equality in $\bbq$ are recursive, so the matrices in
		$\mathrm{SL}_n(\bbq)$ admit a recursive enumeration
		$(g_i)_{i<\omega}$. Let
		$\pi:F_\omega\to\mathrm{SL}_n(\bbq)$ be the epimorphism defined by
		$\pi(a_i)=g_i$. For a reduced word $w\in F_\omega$, exact rational matrix multiplication computes $\pi(w)$ and decides whether $\pi(w)=I_n$. Thus
		$\ker(\pi)$ is recursive, and in particular recursively enumerable. Hence
		$\mathrm{SL}_n(\bbq)$ has a recursive countable presentation.
		
		Also, it's well known that $ \bbz^2 $ is a subgroup. Let
		$d_2=\operatorname{diag}(2,2^{-1},1,\ldots,1)$ and
		$d_3=\operatorname{diag}(3,3^{-1},1,\ldots,1)$.
		These matrices commute, and for $r,s\in\bbz$,
		\[
		d_2^rd_3^s
		=
		\operatorname{diag}
		(2^r3^s,2^{-r}3^{-s},1,\ldots,1).
		\]
		If $d_2^rd_3^s=I_n$, then $2^r3^s=1$, which implies $r=s=0$ by unique factorization in $\bbq^\times$. Thus
		$\langle d_2,d_3\rangle\cong\bbz\times\bbz$.
		Corollary~\ref{cor:recursive-product-obstruction} gives the result.
	\end{proof}

	\section*{Acknowledgments}
	
	The finite-index ascent theorem, Theorem~\ref{thm:finite-index}, was obtained by the author in a one-shot interaction with ChatGPT 5.6 Sol Pro on 16 August 2026 at 13:10 UTC, while investigating the strong topological Rokhlin property for $\mathrm{SL}_2(\bbz)$. The initial prompt and response are reproduced in Appendix~\ref{app:chatgpt-transcript}. The shared conversation, with timestamps, is available at
	\[
	\text{\url{https://chatgpt.com/share/6a84a0c6-d440-83e8-bc75-8e668a1948b2}}.
	\]
	The proof was later reorganized in the globally realizable formalism. The higher-power and free-extension constructions have antecedents in \cite{bitar2024,carrollpenland2015}, and the forms used here are proved above. Other results in the paper were obtained through interactive prompting of ChatGPT 5.6 Sol Pro, and were subsequently checked and rewritten by the author.  All AI-assisted material was reviewed and revised by the author, who takes full responsibility for the mathematical content and the final version of the paper.
	
	Hui Xu subsequently posted \cite{xu2026}; arXiv version~1 was submitted on 17 August 2026 at 12:23:50 UTC. Xu's Theorem~7.3 proves the same finite-index-overgroup statement. Before this, Doucha asked whether \STRP is closed under commensurability and virtual isomorphism, and noted that even the virtually cyclic case was unknown \cite[Question~7.2]{doucha2024}. The interaction of 16 August was conducted without access to, citation of, or reliance on Xu's preprint, which had not yet been publicly posted. Accordingly, the independence claimed here means only that the author obtained the argument through this AI-assisted interaction independently of Xu's preprint; no broader claim about the provenance of the model's output is intended. We did not locate an earlier public source containing the same finite-index theorem.

	\appendix
	\section{Initial ChatGPT prompt and response}\label{app:chatgpt-transcript}
	
	For disclosure and chronology, we reproduce below the initial prompt submitted
	to ChatGPT 5.6 Sol Pro on 16 August 2026 and the corresponding one-shot response.
	The transcript is reproduced from the saved chat record without mathematical
	editing; the original Markdown notation is retained directly in the present source file.
	
	\subsection{Initial prompt}
	
	\phantom{a}
	
	\begin{lstlisting}[style=chattranscript]
		You are to solve the following research-level mathematics problem.
		
		[
		\boxed{\text{Prove that }\mathrm{SL}\_2(\mathbb Z)\text{ has the strong topological Rokhlin property.}}
		]
		
		Use the standard meaning of the strong topological Rokhlin property for countable group actions on the Cantor space: the relevant conjugation action on the Polish space of Cantor actions has a comeager orbit.
		
		Your objective is a complete, rigorous proof.
		
		This prompt deliberately gives you **no proposed proof roadmap**. Do not infer that any particular representation of the group, permanence principle, symbolic-dynamical characterization, generic-action framework, or other technique is intended. Discover the useful structure yourself.
		
		Do not begin your response by proposing a menu of possible approaches or by committing to a proof strategy. Start investigating the mathematics.
		
		Use web search whenever it is useful for checking definitions, locating theorems, determining the current literature status, reading proofs, following citations, or verifying that an apparent lemma is already known. Prefer primary mathematical sources. Verify exact hypotheses rather than relying on search snippets or remembered theorem statements.
		
		Do not ask me to choose a direction, approve a plan, or decide which approach you should pursue. Conduct the proof search autonomously.
		
		## Fundamental standard
		
		The desired output is not:
		
		- a literature survey;
		- a plausible sketch;
		- a reduction to an unproved statement;
		- a collection of possible approaches;
		- evidence that the theorem ought to hold;
		- a proof for a related group;
		- a proof conditional on a permanence property;
		- or a statement that one approach “should adapt.”
		
		The desired output is a proof in which every nontrivial dependency is either proved or cited from a source whose precise hypotheses have been checked.
		
		You may formulate new lemmas and prove them.
		
		You may prove a substantially more general theorem if doing so genuinely simplifies the argument.
		
		You may abandon any approach at any time.
		
		You may return to an abandoned approach if later work provides a genuinely new mechanism.
		
		Do not force the requested conclusion by inventing missing results.
		
		---
		
		## Do not pre-anchor the search
		
		Do not create an initial roadmap such as:
		
		> Approach 1: ...
		>
		> Approach 2: ...
		>
		> Approach 3: ...
		
		and then spend the rest of the investigation inside those categories.
		
		Do not begin by deciding which existing theorem “must” be generalized.
		
		Do not assume that the most obvious algebraic description of the group reveals the correct proof.
		
		Do not assume that a theorem for a related group should extend to this group.
		
		Do not assume that the formulation in which the problem was asked is the formulation in which it is easiest to solve.
		
		The mathematical directions of attack must arise from the investigation itself.
		
		When new information changes the apparent structure of the problem, allow it to generate entirely new approaches that were not contemplated earlier.
		
		---
		
		## Search broadly without becoming shallow
		
		Maintain genuine diversity in the proof search, but do not expose that diversity merely as a list of speculative ideas.
		
		An approach becomes worth retaining only after you have pushed it to concrete mathematics.
		
		For every serious line of investigation, rapidly seek something falsifiable and exact, for example:
		
		- a precisely stated lemma;
		- an explicit construction;
		- an equivalence;
		- a finite or infinitary combinatorial object;
		- an exact extension or lifting problem;
		- a diagram that must commute;
		- an explicit topology or neighborhood calculation;
		- a concrete obstruction;
		- a counterexample to an intermediate conjecture;
		- a verified theorem with exactly matching hypotheses.
		
		Do not count phrases such as
		
		> “one should be able to adapt the argument,”
		
		> “genericity should be preserved,”
		
		> “the finite part should not matter,”
		
		> “this seems to follow from Bass–Serre theory,”
		
		> “one can encode the compatibility condition,”
		
		or
		
		> “a Fraïssé argument may work”
		
		as mathematical progress.
		
		Whenever such a statement appears, immediately turn it into an exact assertion and either prove it, find a reference proving it, disprove it, or classify it as an unresolved bottleneck.
		
		---
		
		## Avoid local minima
		
		Continuously distinguish between two kinds of difficulty.
		
		A **technical bottleneck** is one for which the current approach provides a mechanism and the remaining work consists of carrying that mechanism through carefully.
		
		A **structural bottleneck** is one where the next required statement is essentially a new theorem and the current approach supplies no mechanism for proving it.
		
		Invest deeply in technical bottlenecks.
		
		Do not repeatedly attack a structural bottleneck using cosmetic reformulations of the same idea.
		
		If an approach reaches a structural bottleneck:
		
		1. record exactly what has been proved;
		2. isolate the exact missing implication;
		3. stop spending disproportionate effort on that formulation;
		4. preserve useful lemmas already obtained;
		5. reopen the approach only when another part of the investigation supplies genuinely new information.
		
		Sunk effort is not evidence that an approach deserves further effort.
		
		If several leading approaches converge to the same unresolved assertion, do not automatically interpret this as evidence that the assertion must be proved. It may instead indicate that all those approaches belong to the same local minimum.
		
		At that point deliberately search for a formulation in which that assertion never appears.
		
		---
		
		## Breadth reset
		
		Whenever the investigation has spent substantial effort without reducing the genuine logical distance to the theorem, perform a breadth reset.
		
		A breadth reset does **not** mean writing another list of vague approaches.
		
		It means temporarily discarding the currently dominant conceptual framing and re-examining the problem from its definitions and established facts.
		
		Ask questions such as:
		
		- What exactly would constitute a witness to the desired property?
		- What is the weakest intermediate statement that would already imply the theorem?
		- What parts of the current formulation are accidental?
		- Is there an equivalent formulation in another mathematical language?
		- Can the quantifiers be reorganized?
		- Can one construct the desired generic object directly?
		- Can one characterize failure of the property and rule failure out?
		- Is there a finite approximation principle hidden in the topology?
		- Is there a universal object whose existence would automatically give the result?
		- Is the difficult compatibility condition an artifact of the chosen coordinates?
		- Can the problem be transferred to an object where the group relations are easier to enforce?
		
		These are questions for generating new mathematics, not a list of prescribed approaches.
		
		After a breadth reset, push whichever new ideas arise to concrete lemmas before evaluating them.
		
		Repeat breadth resets whenever necessary.
		
		---
		
		## Preserve independent thinking
		
		Do not let one attractive partial solution determine how every other idea is interpreted.
		
		When two approaches are genuinely different, develop them independently far enough that each reveals its own natural obstruction.
		
		Only afterward compare them.
		
		If two independent investigations unexpectedly produce the same intermediate structure, that is useful evidence that the structure may be intrinsic.
		
		If a difficult intermediate lemma appears only because of one particular formulation and disappears in another, treat that as evidence that the original lemma may be avoidable.
		
		Do not turn an accidental feature of the first promising approach into a mandatory subproblem for the entire proof search.
		
		---
		
		## Literature use
		
		Search the current literature aggressively enough to avoid reproving known results or relying on false folklore.
		
		But do not allow the literature search to become the roadmap.
		
		In particular:
		
		- search for the exact target statement;
		- search alternative terminology;
		- follow references backward;
		- search papers citing the most relevant results;
		- check recent work;
		- inspect proofs, not only abstracts;
		- verify theorem hypotheses and definitions;
		- distinguish published results from remarks, questions, conjectures, and folklore;
		- check whether a purported permanence principle is actually established.
		
		If a known theorem solves the problem immediately, give the proof by correctly applying that theorem.
		
		If the literature stops just short of the target, do not automatically decide that extending the nearest theorem is the correct route. Treat the boundary of the literature as information, not as a prescribed research direction.
		
		---
		
		## Proof-distance discipline
		
		At regular intervals evaluate actual logical distance to the target theorem.
		
		Suppose the current argument has established
		
		[
		A\Longrightarrow B\Longrightarrow C
		]
		
		and the desired result is (T).
		
		Do not describe this as major progress unless the remaining implication
		
		[
		C\Longrightarrow T
		]
		
		is genuinely better understood than the original problem.
		
		A reduction is useful only if the reduced problem has additional exploitable structure.
		
		Be especially suspicious of reductions of the form
		
		[
		T \quad\Longleftarrow\quad L,
		]
		
		where (L) is a newly invented lemma whose proof appears to require essentially everything needed to prove (T).
		
		Such a lemma is a relabeling of the problem, not progress.
		
		---
		
		## Depth discipline
		
		Once an approach contains a genuinely promising mechanism, pursue it deeply.
		
		Do not abandon a route merely because some technical work is lengthy.
		
		Carry constructions through definitions.
		
		Check well-definedness.
		
		Check all group relations.
		
		Check all compatibility conditions.
		
		Check surjectivity when a factor map is used.
		
		Check nonemptiness when a subshift or auxiliary structure is introduced.
		
		Check that locally specified data globalize.
		
		Check all topology: openness, density, (G\_\delta) conditions, continuity, convergence, and conjugacy statements as applicable.
		
		Check the exact quantifiers required by the definition of STRP.
		
		Distinguish carefully between statements that look similar but are logically different—for example, a dense conjugacy class and a comeager conjugacy class.
		
		Whenever a proof invokes a general theorem, explicitly verify every hypothesis in the present setting.
		
		Do not hide the central difficulty inside phrases such as “standard,” “routine,” or “similarly.”
		
		---
		
		## Aggressive falsification
		
		Act as both prover and hostile referee throughout the investigation.
		
		Try to destroy every important intermediate claim before building further on it.
		
		For a proposed lemma:
		
		1. test trivial and degenerate cases;
		2. test finite analogues when meaningful;
		3. examine whether the quantifiers are in the correct order;
		4. check whether the conclusion is actually strong enough for the next step;
		5. search for known counterexamples to similar statements;
		6. examine whether an unstated finiteness, effectiveness, freeness, normality, centrality, compactness, or surjectivity assumption has entered;
		7. verify that the claimed implication concerns the actual group and actual topology in the problem.
		
		If a lemma fails, diagnose the structural reason rather than repairing examples one at a time.
		
		A failed lemma can be valuable if it reveals the invariant or obstruction that a successful proof must control.
		
		---
		
		## Use computation when it increases mathematical information
		
		If finite computation can test a proposed combinatorial construction, finite presentation calculation, automaton, pattern-extension rule, or other finite consequence, use it.
		
		The purpose of computation is primarily to:
		
		- falsify false intermediate statements quickly;
		- reveal missing constraints;
		- discover invariants;
		- test exact conjectures;
		- identify the correct formulation of a lemma.
		
		Do not treat finite experiments as proof of an infinitary assertion.
		
		---
		
		## Cross-pollination
		
		When one investigation produces a useful construction or invariant, ask whether it removes a bottleneck in another investigation.
		
		Transfer concrete mechanisms, not vague optimism.
		
		Examples of meaningful cross-pollination include:
		
		- a finite encoding discovered in one formulation resolving a compatibility problem in another;
		- a topological characterization converting an algebraic construction into genericity;
		- an obstruction theorem identifying precisely what a constructive argument must rule out;
		- an alternative presentation eliminating a relation that complicated a previous construction.
		
		Do not merge approaches merely because they use similar terminology.
		
		---
		
		## Convergence rule
		
		Do not converge onto one proof strategy simply because it currently looks best.
		
		Converge only when one route has acquired a chain of concrete, verified implications whose remaining gaps are technical rather than structural, or when independent lines of investigation provide strong mathematical reasons to focus on the same mechanism.
		
		Even after convergence, periodically ask:
		
		> Is the current hardest lemma genuinely necessary, or only necessary because of the route we chose?
		
		If the answer is unclear, briefly reopen the search space.
		
		---
		
		## No premature stopping
		
		Do not stop because:
		
		- the first natural approach fails;
		- several natural approaches fail;
		- a theorem you hoped to use does not exist;
		- the literature calls a related question open;
		- a plausible permanence property is unknown;
		- the required construction initially appears complicated;
		- the argument encounters a new lemma;
		- or the problem appears harder than expected.
		
		Those events should trigger redirection, not termination.
		
		Continue generating and testing genuinely new mathematical mechanisms.
		
		Do not ask me whether you should continue.
		
		---
		
		## Reliability requirements
		
		Never:
		
		- fabricate a citation;
		- fabricate a theorem number;
		- state an unverified permanence property as known;
		- silently strengthen a theorem;
		- infer a result merely from a finite-index or quotient relationship without justification;
		- confuse evidence with proof;
		- conceal a circular dependency;
		- claim that a construction works before checking the defining relations;
		- call a statement “standard” as a substitute for proving or citing it;
		- present a conjectural lemma as established.
		
		Whenever web sources disagree, inspect the primary sources.
		
		Whenever your memory and a source disagree, verify the source carefully and use the mathematically correct statement.
		
		---
		
		## Final-proof threshold
		
		Do not write the final proof merely because the argument has become plausible.
		
		Before declaring success, construct a dependency graph of the proof at the level of lemmas and external theorems and audit every leaf.
		
		Every leaf must be one of:
		
		1. a definition;
		2. an elementary fact proved in the solution;
		3. a theorem proved in the solution;
		4. an external theorem whose exact statement and hypotheses have been verified.
		
		There must be no leaf labeled:
		
		- “should follow”;
		- “presumably”;
		- “standard adaptation”;
		- “finite extensions should preserve this”;
		- “one expects”;
		- “the same argument works”;
		- “it is natural to believe”;
		- or an equivalent placeholder.
		
		Then perform two audits.
		
		### Forward audit
		
		Read the proposed proof from assumptions to conclusion and challenge every nontrivial implication.
		
		### Backward audit
		
		Begin with
		
		[
		\mathrm{SL}\_2(\mathbb Z)\text{ has STRP}
		]
		
		and recursively trace everything required for that conclusion backward until every dependency terminates in an established fact.
		
		The backward audit is specifically intended to detect a hidden theorem-strength gap buried inside a long argument.
		
		If either audit finds a gap, return to research mode.
		
		---
		
		## What to return if successful
		
		If you obtain a proof that survives both audits, give a clean mathematical exposition rather than a diary of the search process.
		
		Include:
		
		- the theorem;
		- necessary definitions;
		- precisely stated external results with citations;
		- all new lemmas;
		- complete proofs of the new lemmas;
		- the deduction of the theorem;
		- a concise final proof audit identifying the most delicate points.
		
		Do not include discarded approaches unless understanding one of them is necessary for checking the final proof.
		
		---
		
		## What to do if the response limit intervenes before success
		
		Your primary objective remains a complete proof.
		
		If you cannot finish before the current response limit, do not fill the remaining space with general discussion.
		
		Instead give a compact research checkpoint containing only:
		
		1. rigorously established results obtained so far;
		2. exact citations for external results actually used;
		3. the smallest currently unresolved dependencies;
		4. which dependencies have been shown unnecessary or false;
		5. enough exact mathematical state that the investigation can continue in the next message without reconstructing the whole search.
		
		Do not present a speculative roadmap for the continuation.
		
		In the next message, resume active proof search from that state rather than summarizing it again.
		
		---
		
		Begin now.
		
		Do not first give me a plan.
		
		Do not ask me questions.
		
		Do not enumerate possible proof strategies.
		
		Investigate, search, calculate, prove, falsify, redirect, and continue until you have a proof that survives adversarial verification or until the response limit forces a precise mathematical checkpoint.
	\end{lstlisting}
	
	\subsection{Initial response}
	
	\phantom{a}
	
	\begin{lstlisting}[style=chattranscript]
		# Theorem
		
		[
		\boxed{\mathrm{SL}\_{2}(\mathbb Z)\text{ has the strong topological Rokhlin property.}}
		]
		
		Let (C=2^{\mathbb N}) be the Cantor space and let
		
		[
		\operatorname{Act}(G,C)=\operatorname{Hom}\bigl(G,\operatorname{Homeo}(C)\bigr)
		]
		
		carry the subspace topology inherited from (\operatorname{Homeo}(C)^G). The group (\operatorname{Homeo}(C)) acts by conjugation:
		
		[
		(u\cdot \alpha)(g)=u\alpha(g)u^{-1}.
		]
		
		We prove that this action has a comeager orbit when (G=\mathrm{SL}\_{2}(\mathbb Z)).
		
		The sole external input is the following theorem of Kwiatkowska:
		
		> For every (m\geq 1), the diagonal conjugation action of
		> (\operatorname{Homeo}(C)) on (\operatorname{Homeo}(C)^m) has a comeager orbit.
		
		Equivalently, the free group (F\_m) has STRP, since evaluation on a free generating set is an equivariant homeomorphism
		
		[
		\operatorname{Act}(F\_m,C)\cong \operatorname{Homeo}(C)^m.
		]
		
		This equivalence is also explicitly noted in the paper. ([arXiv](https://arxiv.org/pdf/1104.3340 "https://arxiv.org/pdf/1104.3340"))
		
		Thus (F\_2) has STRP. The main point is to prove that STRP ascends from a finite-index subgroup.
		
		## 1. Finite-index ascent
		
		### Proposition 1
		
		Let (G) be a finitely generated countable group and let (H\leq G) have finite index. If (H) has STRP, then (G) has STRP.
		
		### Proof
		
		Put
		
		[
		L=G/H,
		]
		
		the finite set of left cosets.
		
		We first record the neighborhood basis that will be used. For a finite symmetric set (E\subseteq G), a finite clopen partition (\mathcal P) of (C), and (\rho\in\operatorname{Act}(G,C)), let
		
		# [ \mathcal N(\rho;E,\mathcal P)
		
		\left{
			\sigma\in\operatorname{Act}(G,C):
			\begin{array}{l}
				\text{for every }g\in E,\ x\in C,\\
				\sigma(g)x\text{ and }\rho(g)x
				\text{ lie in the same atom of }\mathcal P
			\end{array}
			\right}.
		]
		
		These sets form a neighborhood basis. Indeed, for fixed (g), the condition is equivalent to
		
		[
		\sigma(g)^{-1}(P)=\rho(g)^{-1}(P)
		\qquad(P\in\mathcal P),
		]
		
		and exact image conditions between fixed clopen sets are clopen in
		(\operatorname{Homeo}(C)). Conversely, using the usual uniform metric on homeomorphisms, a partition of sufficiently small mesh controls the forward maps; because (E) is symmetric, it also controls their inverses.
		
		The Boolean algebra of clopen subsets of (C) is countable: every clopen set is a finite union of basic cylinders. Hence there are only countably many finite labeled clopen partitions of (C).
		
		For every clopen partition
		
		[
		\mathcal Q=(Q\_\ell)\_{\ell\in L}
		]
		
		of (C), indexed by (L), with every (Q\_\ell\neq\varnothing), define
		
		# [ U\_{\mathcal Q}
		
		\left{
			\beta\in\operatorname{Act}(G,C):
			\beta(g)Q\_\ell=Q\_{g\ell}
			\text{ for every }g\in G,\ \ell\in L
			\right}.
		]
		
		We establish three facts.
		
		### 1.1. Each (U\_{\mathcal Q}) is clopen
		
		Choose a finite symmetric generating set (S) of (G). It is enough in the definition of (U\_{\mathcal Q}) to require
		
		[
		\beta(s)Q\_\ell=Q\_{s\ell}
		\qquad(s\in S,\ \ell\in L).
		]
		
		For fixed clopen (A,B\subseteq C), the set
		
		[
		{u\in\operatorname{Homeo}(C)\:u(A)=B}
		]
		
		is clopen. Therefore (U\_{\mathcal Q}), being a finite intersection of such conditions, is clopen in (\operatorname{Act}(G,C)).
		
		### 1.2. The union of the sets (U\_{\mathcal Q}) is dense
		
		Let
		
		[
		\rho\in\operatorname{Act}(G,C)
		]
		
		and consider a basic neighborhood
		
		[
		\mathcal N(\rho;E,\mathcal P).
		]
		
		Let (\mathcal R) be the common refinement
		
		# [ \mathcal R
		
		\mathcal P\vee
		\bigvee\_{g\in E}\rho(g)^{-1}\mathcal P.
		]
		
		Thus, whenever (x,y) lie in the same atom of (\mathcal R), they lie in the same atom of (\mathcal P), and for every (g\in E), the points (\rho(g)x,\rho(g)y) lie in the same atom of (\mathcal P).
		
		Consider the action (D) of (G) on (C\times L) given by
		
		[
		D(g)(x,\ell)=(\rho(g)x,g\ell).
		]
		
		For each atom (R\in\mathcal R), choose a homeomorphism
		
		[
		\theta\_R\:R\times L\longrightarrow R.
		]
		
		Such a homeomorphism exists explicitly. Every nonempty clopen subset of (2^{\mathbb N}) is a finite disjoint union of cylinder sets, each homeomorphic to (2^{\mathbb N}); and a finite disjoint union of copies of (2^{\mathbb N}) is again homeomorphic to (2^{\mathbb N}). For example, (C) has a clopen partition into any prescribed finite number of cylinders.
		
		Piece together the maps (\theta\_R) to obtain a homeomorphism
		
		[
		\theta\:C\times L\longrightarrow C
		]
		
		such that
		
		[
		\theta(x,\ell)\in R
		\qquad\text{whenever }x\in R\in\mathcal R.
		]
		
		Transport (D) to an action (\sigma) on (C):
		
		[
		\sigma(g)=\theta D(g)\theta^{-1}.
		]
		
		Set
		
		[
		Q\_\ell=\theta(C\times{\ell}).
		]
		
		Then (\mathcal Q=(Q\_\ell)\_{\ell\in L}) is a clopen partition and
		
		[
		\sigma(g)Q\_\ell=Q\_{g\ell},
		]
		
		so (\sigma\in U\_{\mathcal Q}).
		
		We check that (\sigma\in\mathcal N(\rho;E,\mathcal P)). Write
		
		[
		y=\theta(x,\ell).
		]
		
		By construction, (x) and (y) lie in the same atom of (\mathcal R). Hence, for (g\in E),
		
		[
		\rho(g)x\quad\text{and}\quad \rho(g)y
		]
		
		lie in the same atom of (\mathcal P). Moreover,
		
		# [ \sigma(g)y
		
		\theta(\rho(g)x,g\ell),
		]
		
		and (\theta(\rho(g)x,g\ell)) lies in the same atom of (\mathcal R), hence of (\mathcal P), as (\rho(g)x). Therefore (\sigma(g)y) and (\rho(g)y) lie in the same atom of (\mathcal P). This proves
		
		[
		\sigma\in\mathcal N(\rho;E,\mathcal P).
		]
		
		Consequently,
		
		[
		U:=\bigcup\_{\mathcal Q}U\_{\mathcal Q}
		]
		
		is dense. Since each (U\_{\mathcal Q}) is open, (U) is dense open.
		
		### 1.3. Inside every (U\_{\mathcal Q}), a comeager set belongs to one fixed (G)-orbit
		
		Fix representatives
		
		[
		t\_\ell\in G,\qquad t\_\ell H=\ell,
		]
		
		with (t\_H=e). Define the corresponding cocycle
		
		[
		c(g,\ell)=t\_{g\ell}^{-1}gt\_\ell\in H.
		]
		
		It satisfies
		
		# [ c(g\_1g\_2,\ell)
		
		c(g\_1,g\_2\ell)c(g\_2,\ell).
		\tag{1}
		]
		
		Fix (\mathcal Q). For every (\ell\in L), choose a homeomorphism
		
		[
		\eta\_\ell\:Q\_H\longrightarrow Q\_\ell,
		\qquad \eta\_H=\operatorname{id}\_{Q\_H}.
		]
		
		For (\beta\in U\_{\mathcal Q}), the set (Q\_H) is (H)-invariant. Define
		
		[
		\alpha\_\beta(h)=\beta(h)|\_{Q\_H},
		\qquad h\in H,
		]
		
		and, for (\ell\neq H),
		
		# [ u\_{\beta,\ell}
		
		\eta\_\ell^{-1}\circ
		\beta(t\_\ell)|\_{Q\_H}
		\in\operatorname{Homeo}(Q\_H).
		]
		
		We claim that
		
		[
		\Phi\_{\mathcal Q}\:U\_{\mathcal Q}\longrightarrow
		\operatorname{Act}(H,Q\_H)
		\times
		\operatorname{Homeo}(Q\_H)^{L\setminus{H}},
		]
		
		given by
		
		# [ \Phi\_{\mathcal Q}(\beta)
		
		\left(\alpha\_\beta,(u\_{\beta,\ell})\_{\ell\neq H}\right),
		\tag{2}
		]
		
		is a homeomorphism.
		
		To construct the inverse, let
		
		[
		\alpha\in\operatorname{Act}(H,Q\_H)
		]
		
		and choose (u\_\ell\in\operatorname{Homeo}(Q\_H)) for
		(\ell\neq H). Put
		
		[
		u\_H=\operatorname{id}*{Q\_H},*
		*\qquad*
		*j*\ell=\eta\_\ell u\_\ell\:Q\_H\longrightarrow Q\_\ell.
		]
		
		Define, for (x\in Q\_H),
		
		# [ \beta(g)(j\_\ell x)
		
		j\_{g\ell},\alpha(c(g,\ell))x.
		\tag{3}
		]
		
		Since the (Q\_\ell)'s form a finite clopen partition, this defines a homeomorphism of (C). By (1),
		
		[
		\begin{aligned}
			\beta(g\_1)\beta(g\_2)(j\_\ell x)
			&=
			j\_{g\_1g\_2\ell}
			\alpha(c(g\_1,g\_2\ell))
			\alpha(c(g\_2,\ell))x\\
			&=
			j\_{g\_1g\_2\ell}
			\alpha(c(g\_1g\_2,\ell))x\\
			&=
			\beta(g\_1g\_2)(j\_\ell x).
		\end{aligned}
		]
		
		Thus (\beta) is a (G)-action and belongs to (U\_{\mathcal Q}).
		
		For (h\in H),
		
		[
		c(h,H)=h,
		]
		
		so (\beta(h)|\_{Q\_H}=\alpha(h)). Furthermore,
		
		[
		c(t\_\ell,H)=e,
		]
		
		so
		
		[
		\beta(t\_\ell)|*{Q\_H}=j*\ell=\eta\_\ell u\_\ell.
		]
		
		Hence (3) is indeed inverse to (2). Conversely, starting with
		(\beta\in U\_{\mathcal Q}), we have
		
		[
		j\_\ell=\beta(t\_\ell)|\_{Q\_H},
		]
		
		and therefore, for (x\in Q\_H),
		
		[
		\begin{aligned}
			\beta(g)(j\_\ell x)
			&=\beta(gt\_\ell)x\\
			&=\beta(t\_{g\ell}c(g,\ell))x\\
			&=j\_{g\ell}\alpha\_\beta(c(g,\ell))x.
		\end{aligned}
		]
		
		The displayed formulas also prove continuity in both directions: for fixed (g), the restriction of (\beta(g)) to each clopen (Q\_\ell) is obtained by finitely many compositions and inversions in the relevant homeomorphism groups. This proves that (\Phi\_{\mathcal Q}) is a homeomorphism.
		
		Now let
		
		[
		\alpha\_\ast\in\operatorname{Act}(H,C)
		]
		
		have a comeager (H)-conjugacy orbit. Since (Q\_H) is a nonempty clopen Cantor space, transport this orbit to a comeager orbit
		
		[
		\mathcal O\_{\mathcal Q}
		\subseteq
		\operatorname{Act}(H,Q\_H).
		]
		
		Set
		
		# [ V\_{\mathcal Q}
		
		\Phi\_{\mathcal Q}^{-1}
		\left(
		\mathcal O\_{\mathcal Q}
		\times
		\operatorname{Homeo}(Q\_H)^{L\setminus{H}}
		\right).
		]
		
		Then (V\_{\mathcal Q}) is comeager in (U\_{\mathcal Q}).
		
		We next show that all the sets (V\_{\mathcal Q}), as (\mathcal Q) varies, are contained in one (G)-conjugacy orbit.
		
		Define the induced action (I\_\ast) of (G) on (L\times C) by
		
		# [ I\_\ast(g)(\ell,x)
		
		\bigl(g\ell,\alpha\_\ast(c(g,\ell))x\bigr).
		\tag{4}
		]
		
		The cocycle identity (1) shows that this is an action. Since (L) is finite and nonempty, (L\times C) is homeomorphic to (C); thus we may regard (I\_\ast) as a Cantor action of (G).
		
		Take (\beta\in V\_{\mathcal Q}). Write
		
		# [ \Phi\_{\mathcal Q}(\beta)
		
		\left(\alpha\_\beta,(u\_{\beta,\ell})\_{\ell\neq H}\right)
		]
		
		and put
		
		[
		j\_\ell=\eta\_\ell u\_{\beta,\ell}.
		]
		
		Since (\alpha\_\beta\in\mathcal O\_{\mathcal Q}), there is a homeomorphism
		
		[
		\varphi\:C\longrightarrow Q\_H
		]
		
		such that
		
		# [ \alpha\_\beta(h)\varphi
		
		\varphi\alpha\_\ast(h)
		\qquad(h\in H).
		]
		
		Define
		
		[
		K\:L\times C\longrightarrow C,
		\qquad
		K(\ell,x)=j\_\ell\varphi(x).
		]
		
		Then
		
		[
		\begin{aligned}
			\beta(g)K(\ell,x)
			&=
			\beta(g)j\_\ell\varphi(x)\\
			&=
			j\_{g\ell}
			\alpha\_\beta(c(g,\ell))\varphi(x)\\
			&=
			j\_{g\ell}
			\varphi\alpha\_\ast(c(g,\ell))x\\
			&=
			K I\_\ast(g)(\ell,x).
		\end{aligned}
		]
		
		Thus (\beta) is conjugate to (I\_\ast). Hence every (V\_{\mathcal Q}) is contained in the same (G)-conjugacy orbit, say (\mathcal O\_G).
		
		Finally,
		
		[
		\operatorname{Act}(G,C)\setminus\mathcal O\_G
		\subseteq
		\left(\operatorname{Act}(G,C)\setminus U\right)
		\cup
		\bigcup\_{\mathcal Q}
		\left(U\_{\mathcal Q}\setminus V\_{\mathcal Q}\right).
		\tag{5}
		]
		
		The first set on the right is nowhere dense because (U) is dense open. For each (\mathcal Q), the set
		(U\_{\mathcal Q}\setminus V\_{\mathcal Q}) is meager in the clopen subspace (U\_{\mathcal Q}), and hence is meager in the entire action space. There are only countably many (\mathcal Q). Therefore the right-hand side of (5) is meager.
		
		Consequently (\mathcal O\_G) is comeager, and (G) has STRP. ∎
		
		## 2. A finite-index free subgroup of (\mathrm{SL}\_2(\mathbb Z))
		
		Set
		
		[
		A=
		\begin{pmatrix}
			1&2\\
			0&1
		\end{pmatrix},
		\qquad
		B=
		\begin{pmatrix}
			1&0\\
			2&1
		\end{pmatrix},
		\qquad
		H=\langle A,B\rangle.
		]
		
		We prove that (H\cong F\_2) and that
		
		[
		[\mathrm{SL}\_2(\mathbb Z)\:H]=12.
		]
		
		### Proposition 2
		
		The matrices (A,B) freely generate a free group of rank (2).
		
		### Proof
		
		Let the matrices act on the real projective line
		
		[
		\mathbb P^1(\mathbb R)=\mathbb R\cup{\infty}
		]
		
		by Möbius transformations. Write (a,b) for the projective transformations induced by (A,B). Then
		
		[
		a(x)=x+2,
		\qquad
		b(x)=\frac{x}{2x+1}.
		]
		
		More generally, for every (n\in\mathbb Z),
		
		[
		A^n=
		\begin{pmatrix}
			1&2n\\
			0&1
		\end{pmatrix},
		\qquad
		B^n=
		\begin{pmatrix}
			1&0\\
			2n&1
		\end{pmatrix},
		]
		
		so
		
		[
		a^n(x)=x+2n,
		\qquad
		b^n(x)=\frac{x}{2nx+1},
		\qquad
		b^n(\infty)=\frac1{2n}
		\quad(n\neq 0).
		]
		
		Let
		
		[
		X=(-1,1),
		\qquad
		Y={x\in\mathbb R:|x|>1}\cup{\infty}.
		]
		
		These sets are disjoint. For (n\neq0),
		
		[
		a^n(X)\subseteq Y.
		\tag{6}
		]
		
		Also,
		
		[
		b^n(Y)\subseteq X.
		\tag{7}
		]
		
		Indeed, if (|x|>1), then
		
		[
		|2nx+1|
		\geq 2|n||x|-1
		\geq 2|x|-1
		
		> |x|,
		> ]
		
		and therefore
		
		[
		\left|\frac{x}{2nx+1}\right|<1.
		]
		
		Moreover (|1/(2n)|<1).
		
		We now verify directly that no nontrivial reduced word in (A,B) acts trivially projectively. Suppose first that the first and last syllables of a reduced word (w) belong to the same cyclic factor.
		
		If both are powers of (A), then, reading the word from right to left and using (6)–(7), one obtains
		
		[
		w(X)\subseteq Y,
		]
		
		so (w) is not the identity. If both are powers of (B), then similarly
		
		[
		w(Y)\subseteq X.
		]
		
		If the first and last syllables belong to different factors, suppose for example that the last syllable is (A^r). Choose (m\neq0,-r). Then the reduced form of
		
		[
		A^{-m}wA^m
		]
		
		begins and ends with a nonzero power of (A), so this conjugate is not the identity by the preceding case. Hence (w) itself is not the identity. The case in which the last syllable is a power of (B) is identical.
		
		Thus the projective transformations induced by (A,B) freely generate a free group. In particular, the matrices (A,B) themselves freely generate a free group:
		
		[
		H\cong F\_2.
		\tag{8}
		]
		
		The same argument also gives
		
		[
		H\cap{\pm I}={I},
		\tag{9}
		]
		
		because both (I) and (-I) act trivially on the projective line, while no nontrivial word in (A,B) does. ∎
		
		### Proposition 3
		
		Let
		
		# [ \Gamma(2)
		
		\ker\left(
		\mathrm{SL}\_2(\mathbb Z)
		\longrightarrow
		\mathrm{SL}\_2(\mathbb F\_2)
		\right).
		]
		
		Then
		
		[
		\Gamma(2)=H\langle -I\rangle
		]
		
		and
		
		[
		[\Gamma(2)\:H]=2.
		]
		
		### Proof
		
		Certainly (A,B\in\Gamma(2)), so (H\leq\Gamma(2)).
		
		Let
		
		[
		M=
		\begin{pmatrix}
			a&b\\
			c&d
		\end{pmatrix}
		\in\Gamma(2).
		]
		
		Then
		
		[
		a,d\text{ are odd},
		\qquad
		b,c\text{ are even},
		]
		
		and
		
		[
		\gcd(a,c)=1,
		\tag{10}
		]
		
		because every common divisor of (a,c) divides (ad-bc=1).
		
		Left multiplication by (A^n) or (B^n) changes the first column as follows:
		
		# [ A^n \binom ac
		
		# \binom{a+2nc}{c}, \qquad B^n \binom ac
		
		\binom a{c+2na}.
		\tag{11}
		]
		
		Suppose (c\neq0). Since (a) is odd and (c) is even,
		
		[
		|a|\neq |c|.
		]
		
		If (|c|>|a|), choose (n\in\mathbb Z) nearest to (-c/(2a)). Then
		
		[
		|c+2na|\leq |a|.
		]
		
		The left side is even and (|a|) is odd, so equality is impossible. Thus
		
		[
		|c+2na|<|a|.
		]
		
		After multiplication by (B^n), the maximum of the absolute values of the two first-column entries has strictly decreased.
		
		If (|a|>|c|), choose (n) nearest to (-a/(2c)). Then
		
		[
		|a+2nc|\leq |c|.
		]
		
		Now the left side is odd and (|c|) is even, so again equality is impossible:
		
		[
		|a+2nc|<|c|.
		]
		
		After multiplication by (A^n), the same maximum has strictly decreased.
		
		Repeating this process must terminate. The first entry remains odd, so it can never become (0). Therefore we eventually obtain some (W\in H) such that the first column of (WM) is
		
		[
		\binom{\varepsilon}{0},
		\qquad \varepsilon\in{\pm1};
		]
		
		here (10) gives (|\varepsilon|=1).
		
		Because (WM\in\Gamma(2)) and has determinant (1), it has the form
		
		[
		WM=
		\begin{pmatrix}
			\varepsilon&q\\
			0&\varepsilon
		\end{pmatrix},
		\qquad q\in2\mathbb Z.
		]
		
		If (\varepsilon=1), then
		
		[
		WM=A^{q/2}.
		]
		
		If (\varepsilon=-1), then
		
		[
		WM=-A^{-q/2}.
		]
		
		It follows that
		
		[
		M\in H\langle -I\rangle.
		]
		
		Thus
		
		[
		\Gamma(2)=H\langle -I\rangle.
		]
		
		Since (-I) is central and (9) gives
		(H\cap\langle -I\rangle={I}), the two right cosets are (H) and (H(-I)). Therefore
		
		[
		[\Gamma(2)\:H]=2.
		\tag{12}
		]
		
		∎
		
		### Proposition 4
		
		[
		[\mathrm{SL}\_2(\mathbb Z)\:H]=12.
		]
		
		### Proof
		
		Over (\mathbb F\_2), every nonzero determinant equals (1), so
		
		[
		\mathrm{SL}\_2(\mathbb F\_2)=\mathrm{GL}\_2(\mathbb F\_2).
		]
		
		The latter group has order
		
		[
		(2^2-1)(2^2-2)=3\cdot2=6.
		]
		
		Reduction modulo (2) is surjective. Indeed, the reductions of
		
		[
		S=
		\begin{pmatrix}
			0&-1\\
			1&0
		\end{pmatrix},
		\qquad
		T=
		\begin{pmatrix}
			1&1\\
			0&1
		\end{pmatrix}
		]
		
		act on the three nonzero vectors of (\mathbb F\_2^2) as two distinct transpositions:
		
		[
		\overline S:(1,0)\leftrightarrow(0,1),
		]
		
		fixing ((1,1)), while
		
		[
		\overline T:(0,1)\leftrightarrow(1,1),
		]
		
		fixing ((1,0)). These transpositions generate the full symmetric group on the three nonzero vectors. The action of
		(\mathrm{GL}\_2(\mathbb F\_2)) on those vectors is faithful, so the reductions of (S,T) generate all of
		(\mathrm{SL}\_2(\mathbb F\_2)).
		
		Consequently,
		
		[
		[\mathrm{SL}\_2(\mathbb Z):\Gamma(2)]=6.
		\tag{13}
		]
		
		Combining (12) and (13),
		
		# [ [\mathrm{SL}\_2(\mathbb Z)\:H]
		
		# [\mathrm{SL}\_2(\mathbb Z):\Gamma(2)] [\Gamma(2)\:H]
		
		# 6\cdot2
		
		12.
		
		]
		
		∎
		
		## 3. Conclusion
		
		By Proposition 2,
		
		[
		H=\langle A,B\rangle\cong F\_2.
		]
		
		Kwiatkowska’s theorem therefore implies that (H) has STRP.
		
		By Proposition 4, (H) has finite index (12) in
		
		[
		G=\mathrm{SL}\_2(\mathbb Z).
		]
		
		Moreover (G) is finitely generated: a finite generating set for (H), together with a finite set of coset representatives for (G/H), generates (G).
		
		Proposition 1 now applies and gives that (G) has STRP. Hence the conjugation action
		
		[
		\operatorname{Homeo}(C)\curvearrowright
		\operatorname{Act}\bigl(\mathrm{SL}\_2(\mathbb Z),C\bigr)
		]
		
		has a comeager orbit.
		
		[
		\boxed{\mathrm{SL}\_{2}(\mathbb Z)\text{ has the strong topological Rokhlin property.}}
		]
		
		## Proof audit
		
		The dependency graph is
		
		[
		\text{Kwiatkowska's ample-generics theorem}
		\Longrightarrow
		F\_2\text{ has STRP}
		]
		
		together with
		
		[
		H=\langle A,B\rangle\cong F\_2,
		\qquad
		[\mathrm{SL}\_2(\mathbb Z)\:H]=12,
		]
		
		and
		
		[
		\text{finite-index ascent}
		\Longrightarrow
		\mathrm{SL}\_2(\mathbb Z)\text{ has STRP}.
		]
		
		The delicate points are all accounted for:
		
		1. **Openness in the ascent argument:** finite generation is used exactly to make each (U\_{\mathcal Q}) a finite intersection of clopen conditions.
		2. **Density:** the product action (\rho\times(G\curvearrowright G/H)) is conjugated back to (C) by a homeomorphism that stays inside the atoms of a sufficiently fine clopen refinement, giving the required neighborhood calculation.
		3. **Single conjugacy class:** the chart (2) and inverse formula (3) show that all transition maps between coset pieces are inessential conjugacy coordinates; once the (H)-restriction is generic, the whole (G)-action is conjugate to the fixed induced action (4).
		4. **The subgroup calculation:** ping-pong handles every reduced word, including words whose first and last syllables lie in different factors; the descent inside (\Gamma(2)) is strictly decreasing because one coordinate is odd and the other even; and reduction modulo (2) is explicitly shown to be surjective.
		
		Thus every leaf of the argument is either proved above or is the precisely cited ample-generics theorem.
	\end{lstlisting}

\begin{thebibliography}{99}
		
		
		\bibitem{andre2021}
		Simon Andr\'e,
		\emph{Elementary subgroups of virtually free groups},
		Groups Geom. Dyn. \textbf{15} (2021), no.~4, 1523--1552,
		doi:10.4171/GGD/638.
		
		\bibitem{barbiericarrasco2024}
		Sebasti\'an Barbieri and Nicanor Carrasco-Vargas,
		\emph{Medvedev degrees of subshifts on groups},
		arXiv:2406.12777v2.
		
		
		\bibitem{bitar2024}
		Nicol\'as Bitar,
		\emph{Realizability of subgroups by subshifts of finite type},
		Groups Geom. Dyn. (2025), published online first,
		doi:10.4171/GGD/900,
		arXiv:2406.04132.
		
		\bibitem{carrollpenland2015}
		David Carroll and Andrew Penland,
		\emph{Periodic points on shifts of finite type and commensurability invariants of groups},
		New York J. Math. \textbf{21} (2015), 811--822.
		
		\bibitem{carson2012}
		J.~Carson, V.~Harizanov, J.~Knight, K.~Lange, C.~McCoy, A.~Morozov,
		S.~Quinn, C.~Safranski, and J.~Wallbaum,
		\emph{Describing free groups},
		Trans. Amer. Math. Soc. \textbf{364} (2012), no.~11, 5715--5728,
		doi:10.1090/S0002-9947-2012-05456-0.
		
		\bibitem{carrasco2026}
		Nicanor Carrasco-Vargas,
		\emph{The strong topological Rokhlin property and Medvedev degrees of SFTs},
		arXiv:2601.03501v1.
		
		\bibitem{dahmaniyaman2008}
		Fran\c{c}ois Dahmani and Asli Yaman,
		\emph{Symbolic dynamics and relatively hyperbolic groups},
		Groups Geom. Dyn. \textbf{2} (2008), no.~2, 165--184,
		doi:10.4171/GGD/35.
		
		\bibitem{doucha2024}
		Michal Doucha,
		\emph{Strong topological Rokhlin property, shadowing, and symbolic dynamics of countable groups},
		J. Eur. Math. Soc., published online 18 December 2024,
		doi:10.4171/JEMS/1584,
		arXiv:2211.08145.
		
		\bibitem{douchamelleraytsankov2026}
		Michal Doucha, Julien Melleray, and Todor Tsankov,
		\emph{Dense and comeager conjugacy classes in zero-dimensional dynamics},
		arXiv:2507.05474v2.
		
		\bibitem{kechris1995}
		Alexander S. Kechris,
		\emph{Classical descriptive set theory},
		Graduate Texts in Mathematics, vol.~156,
		Springer-Verlag, New York, 1995,
		doi:10.1007/978-1-4612-4190-4.
		
		\bibitem{kechrisrosendal2007}
		Alexander S. Kechris and Christian Rosendal,
		\emph{Turbulence, amalgamation, and generic automorphisms of homogeneous structures},
		Proc. Lond. Math. Soc. (3) \textbf{94} (2007), no.~2, 302--350,
		doi:10.1112/plms/pdl007.
		
		\bibitem{kwiatkowska2012}
		Aleksandra Kwiatkowska,
		\emph{The group of homeomorphisms of the Cantor set has ample generics},
		Bull. Lond. Math. Soc. \textbf{44} (2012), no.~6, 1132--1146,
		doi:10.1112/blms/bds039.
		
		\bibitem{lyndonschupp1977}
		Roger C. Lyndon and Paul E. Schupp,
		\emph{Combinatorial group theory},
		Ergebnisse der Mathematik und ihrer Grenzgebiete, vol.~89,
		Springer-Verlag, Berlin--New York, 1977,
		doi:10.1007/978-3-642-61896-3.
		
		\bibitem{minasyan2026}
		Ashot Minasyan,
		\emph{Property (LR) and an embedding theorem for virtually free groups},
		arXiv:2603.17596v2.
		
		
		
		
		\bibitem{xu2026}
		Hui Xu,
		\emph{Cayley-tree pseudo-orbit tracing: period subgroups and relative geometry},
		arXiv:2608.16483v1.
		
		
		
		
	\end{thebibliography}
\end{document}